\documentclass[11pt]{amsart}

\usepackage{a4wide}
\usepackage{amssymb}
\usepackage{mathrsfs}
\usepackage{enumerate}
\usepackage{esint}
\usepackage{titletoc}
\usepackage[colorlinks=true, urlcolor=blue, linkcolor=blue, citecolor=black]{hyperref}
\usepackage[nameinlink]{cleveref}
\usepackage{graphicx}
\usepackage{todonotes}

\theoremstyle{plain}
\newtheorem{theorem}{Theorem}[section]
\newtheorem{lemma}[theorem]{Lemma}
\newtheorem{corollary}[theorem]{Corollary}
\newtheorem{proposition}[theorem]{Proposition}

\theoremstyle{definition}
\newtheorem{definition}[theorem]{Definition}
\newtheorem{remark}[theorem]{Remark}
\newtheorem*{remark*}{Remark}

\numberwithin{equation}{section}

\renewcommand{\d}{\textnormal{d}}
\newcommand{\dx} {\,\mathrm{d}x}
\newcommand{\dy} {\,\mathrm{d}y}

\newcommand{\ds} {\,\mathrm{d}s}
\newcommand{\R}{\mathbb{R}}
\newcommand{\Rd}{\mathbb{R}^d}
\DeclareMathOperator{\supp}{supp}
\newcommand{\sign}{\text{sign}}

\makeatletter
\@namedef{subjclassname@2020}{\textup{2020} Mathematics Subject Classification}
\makeatother

\title[Regularity for nonlocal problems with non-standard growth]{Regularity for nonlocal problems with non-standard growth}

\author{Jamil Chaker}
\address{Fakult\"at f\"ur Mathematik, Universit\"at Bielefeld, 33615 Bielefeld, Germany}
\email{jchaker@math.uni-bielefeld.de}

\author{Minhyun Kim}
\address{Fakult\"at f\"ur Mathematik, Universit\"at Bielefeld, 33615 Bielefeld, Germany}
\email{minhyun.kim@uni-bielefeld.de}

\author{Marvin Weidner}
\address{Fakult\"at f\"ur Mathematik, Universit\"at Bielefeld, 33615 Bielefeld, Germany}
\email{mweidner@math.uni-bielefeld.de}

\thanks{Jamil Chaker gratefully acknowledges funding by the Deutsche Forschungsgemeinschaft (SFB 1283/2 2021 - 317210226). Minhyun Kim and Marvin Weidner gratefully acknowledge funding by the Deutsche Forschungsgemeinschaft (GRK 2235/2 2021 - 282638148).}

\begin{document}

\begin{abstract}
We study robust regularity estimates for local minimizers of nonlocal functionals with non-standard growth of $(p,q)$-type and for weak solutions to a related class of nonlocal equations. The main results of this paper are local boundedness and H\"older continuity of minimizers and weak solutions. 
Our approach is based on the study of corresponding De Giorgi classes.
\end{abstract}

\maketitle

\section{Introduction} \label{sec:introduction}

The aim of this paper is to prove regularity properties of local minimizers and weak solutions to a class of nonlocal problems with non-standard growth.

Let $s \in (0,1)$, $\Lambda \geq 1$ and $\Omega \subset \Rd$ be an open set. We study energy functionals of the form
\begin{equation} \label{eq:nonlocalfunctional}
u \mapsto \mathcal{I}_f(u) = (1-s)\iint_{(\Omega^c \times \Omega^c)^c} f\left( \frac{\vert u(x) - u(y) \vert}{\vert x-y \vert^s} \right) \frac{k(x,y)}{|x-y|^{d}} \dy \dx,
\end{equation}
where $f : [0,\infty) \to [0,\infty)$ is a convex increasing function and $k: \Rd \times \Rd \to \R$ is a measurable function satisfying
\begin{equation} \label{eq:k} \tag{$k$}
k(x,y) = k(y,x) \quad\text{and}\quad \Lambda^{-1} \leq k(x,y) \leq \Lambda \quad\text{for a.e. } x, y \in \Rd.
\end{equation}

When $f(t) = t^p$ with $p>1$ and $\Lambda =1$, the functional \eqref{eq:nonlocalfunctional} becomes the standard fractional $p$-functional whose corresponding operator is the fractional $p$-Laplacian. The regularity theory for this case is well established, see \cite{DiKuPa14,DiKuPa16,Coz17}.

Apparently the functional \eqref{eq:nonlocalfunctional} is governed by the function $f$, which controls the growth behavior for large and small values of $|u(x)-u(y)||x-y|^{-s}$. To establish the regularity theory, we need some growth conditions on $f$.

Let $1 \leq p \leq q$. We say that $f$ satisfies \eqref{eq:pq} if $f$ is differentiable and satisfies for all $t\geq 0$
\begin{subequations}
\makeatletter
\def\@currentlabel{$f_{p}^{q}$}
\makeatother
\label{eq:pq}
\begin{align}
p f(t) \leq ~&t f'(t), \label{eq:pq-lower} \tag{$f_{p}$} \\
& t f'(t) \le qf(t). \label{eq:pq-upper} \tag{$f^{q}$}
\end{align}
\end{subequations}
Condition \eqref{eq:pq} can be interpreted as a {\it$(p,q)$-growth condition} since it implies that
\setcounter{equation}{1}
\begin{equation}
\label{eq:pq-growth}
f(1)(t^p-1) \leq f(t) \leq f(1)(t^q+1).
\end{equation}
For a detailed discussion of \eqref{eq:pq} we refer to \Cref{sec:preliminaries}.

To study local boundedness of minimizers, we work under the assumption that there exists a constant $c_0 > 0$ such that for all $t \geq 0$
\begin{equation} \label{eq:non-degeneracy} \tag{$f\gtrsim t^p$}
c_0 t^p \leq f(t).
\end{equation}
Throughout the paper, we will assume without loss of generality that $f(0)=0$ and $f(1)=1$. This is possible since a function $u$ minimizes $\mathcal{I}_{f}$ if and only if $u$ minimizes $\mathcal{I}_{(f - f(0))/(f(1)-f(0))}$. 

In the following, we present the first main result of this paper. It is concerned with H\"older estimates and local boundedness for local minimizers of \eqref{eq:nonlocalfunctional}.

\begin{theorem}[Local minimizers]
\label{thm:minimizer}
Let $s_0 \in (0,1)$, $1 < p \le q$, $\Lambda \geq 1$, $c_0 > 0$ and assume $s \in [s_0, 1)$. Let $f: [0,\infty) \to [0,\infty)$ be a convex increasing function satisfying \eqref{eq:pq-upper} and let $k: \Rd \times \Rd \to \R$ be a measurable function satisfying \eqref{eq:k}. Let $u \in V^{s,f}(\Omega | \Rd)$ be a local minimizer of \eqref{eq:nonlocalfunctional}.
\begin{enumerate} [(i)]
\item
Assume that $f$ satisfies \eqref{eq:pq-lower}. Then, there exist $\alpha \in (0,1)$ and $C > 0$, depending on $d$, $s_0$, $p$, $q$ and $\Lambda$, such that for any $B_{8R}(x_0) \subset \Omega$
\begin{equation} \label{eq:Holder}
R^\alpha [u]_{C^{\alpha}(\overline{B_R(x_0)})} \leq C \|u\|_{L^{\infty}(B_{4R}(x_0))} + \mathrm{Tail}_{f'}(u; x_0, 4R).
\end{equation}
\item
Assume that $sp<d$, $q < p^{\ast}:=dp/(d-sp)$ and that $f$ satisfies \eqref{eq:non-degeneracy}. Then, $u \in L^{\infty}_{\mathrm{loc}}(\Omega)$. Moreover, for each $B_{2R}(x_0) \subset \Omega$ there exists $C > 0$, depending on $d$, $s_0$, $p$, $q$, $p^{\ast}-q$, $\Lambda$, $c_0$ and $R$, such that for every $\delta \in (0,1)$
\begin{equation} \label{eq:local-boundedness}
\sup_{B_{R}(x_0)} |u| \leq \delta \mathrm{Tail}_{f'}(u; x_0, R) + C \delta^{-(q-1)\frac{p^{\ast}}{p} \frac{1}{p^{\ast}-q}} \left( \fint_{B_R} |u(x)|^q \dx \right)^{\frac{1}{p}\frac{p^{\ast}-p}{p^{\ast}-q}} + \delta^{\frac{q-1}{q}}.
\end{equation}
\end{enumerate}
\end{theorem}

We refer to \Cref{sec:degiorgiclasses} for the definition of the function space $V^{s, f}(\Omega|\Rd)$ and the tail term $\mathrm{Tail}_{f'}$. The proof of \Cref{thm:minimizer} and the definition of a local minimizer are given in \Cref{sec:minimizer}.

The second main result of this paper is concerned with weak solutions to a related class of nonlocal equations. To motivate our result, we first point out that the Euler--Lagrange equation corresponding to the functional \eqref{eq:nonlocalfunctional} is given by
\begin{equation}
\label{eq:eulerlagrange}
(1-s) \text{p.v.} \int_{\Rd} f'\left( \frac{|u(x)-u(y)|}{|x-y|^s}\right)\frac{u(x)-u(y)}{\vert u(x) - u(y) \vert} \frac{k(x,y)}{|x-y|^{d+s}} \dy = 0 \quad\text{in }\Omega.
\end{equation}
For convex differentiable functions $f$, it is well-known that weak solutions to the Euler--Lagrange equation are minimizers of the functional \eqref{eq:nonlocalfunctional}.
In this article, we consider a more general class of equations
\begin{equation} \label{eq:nonlocalequation}
\mathcal{L}_{h}u = 0 \quad\text{in } \Omega
\end{equation}
with nonlocal operators of the form
\begin{equation} \label{eq:nonlocaloperator}
\mathcal{L}_{h}u(x) = (1-s) \text{p.v.} \int_{\Rd} h\left( x, y, \frac{u(x)-u(y)}{|x-y|^s}\right) \frac{\dy}{|x-y|^{d+s}},
\end{equation}
where $h : \Rd \times \Rd \times \R \to \R$ is a measurable function satisfying the structure condition
\begin{equation}\label{eq:h} \tag{$h$}
h(x, y, t) = h(y, x, t), \quad \sign(t)\frac{1}{\Lambda} f'(|t|) \leq h(x,y,t) \leq \Lambda f'(|t|)
\end{equation}
for a.e. $x, y \in \Rd$ and for all $t \in \R$, and $f : [0,\infty) \to [0,\infty)$ is convex, increasing, differentiable and satisfies $f(0) = 0$, $f(1) = 1$.\\
Note that in the special case $h(x,y,t) = \sign(t)f'(|t|)k(x,y)$ for some $k$ satisfying \eqref{eq:k}, the equations \eqref{eq:eulerlagrange} and \eqref{eq:nonlocalequation} coincide.

We are ready to state the second main result of this article, which establishes H\"older estimates and local boundedness for weak solution to \eqref{eq:nonlocalequation}.

\begin{theorem}[Weak solutions]
\label{thm:weaksol}
Let $s_0 \in (0,1)$, $1 < p \le q$, $\Lambda \geq 1$, $c_0 > 0$ and assume $s \in [s_0, 1)$. Let $f: [0,\infty) \to [0,\infty)$ be a convex increasing function satisfying \eqref{eq:pq-upper} and let $h: \Rd \times \Rd \times \R \to \R$ be a measurable function satisfying \eqref{eq:h}. Let $u \in V^{s,f}(\Omega | \Rd)$ be a weak solution to \eqref{eq:nonlocalequation}.
\begin{enumerate} [(i)]
\item
Assume that $f$ satisfies \eqref{eq:pq-lower}. Then, there exist $\alpha \in (0,1)$ and $C > 0$, depending on $d$, $s_0$, $p$, $q$ and $\Lambda$, such that for any $B_{8R}(x_0) \subset \Omega$ the estimate \eqref{eq:Holder} holds.
\item
Assume that $sp < d$, $q < p^{\ast}$ and that $f$ satisfies \eqref{eq:non-degeneracy}. Then, $u \in L^{\infty}_{\mathrm{loc}}(\Omega)$. Moreover, for each $B_{2R}(x_0) \subset \Omega$ there exists $C > 0$, depending on $d$, $s_0$, $p$, $q$, $p^{\ast}-q$, $\Lambda$, $c_0$ and $R$, such that for every $\delta \in (0,1)$ the estimate \eqref{eq:local-boundedness} holds.
\end{enumerate}
\end{theorem}

The proof of \Cref{thm:weaksol} and the definition of a weak solution are given in \Cref{sec:weaksoln}. Note that we can assume without loss of generality that $f(0) = 0$, $f(1) = 1$ because a function $u$ solves $\mathcal{L}_h u = 0$ if and only if it solves $\mathcal{L}_{h/f(1)} u = 0$.

In fact, we prove H\"older estimates and local boundedness for functions in De Giorgi classes (see \Cref{sec:degiorgiclasses}). The corresponding results, \Cref{thm:Fholder} and \Cref{thm:locB}, are more general. \Cref{thm:minimizer} and \Cref{thm:weaksol} follow by the observation that minimizers, as well as weak solutions, belong to the corresponding De Giorgi classes, see \Cref{sec:minimizer} and \Cref{sec:weaksoln}.

\begin{remark}
Our results are robust in the sense that the constants $C$ and $\alpha$ stay uniform as $s \to 1^-$, since they depend only on $s_0$, not on $s$.
\begin{enumerate}
\item \Cref{thm:minimizer} generalizes the results in \cite{MoNa91} to nonlocal functionals. 
We work under the same assumptions on the growth function $f$ as in that article. In this sense, the assumptions on $f$ used in our results are natural.
\item \Cref{thm:weaksol} can be linked to the paper \cite{Lie91}. Some assumptions on the regularity and growth of the function $f$ in \cite{Lie91} are more restrictive than our assumptions, but in return allow Lieberman to prove $C^{1,\beta}$ regularity of weak solutions to the Euler--Lagrange equations. 
\end{enumerate}
\end{remark}

Our approach for studying local minimizers and weak solutions is based on so-called De Giorgi classes. We show that minimizers of \eqref{eq:nonlocalfunctional} and weak solutions to \eqref{eq:nonlocalequation} satisfy a suitable improved fractional Caccioppoli inequality, from which the definition of the De Giorgi class emerges. This inequality together with an isoperimetric-type inequality allow us to deduce the H\"older estimates for locally bounded minimizers and weak solutions following the methods from \cite{MoNa91} and \cite{Coz17}. We emphasize that there is no restriction on the gap between $p$ and $q$ for the H\"older estimates. Furthermore, we derive the local boundedness of functions in De Giorgi classes under the assumption that $1<p\leq q<p^{\ast}$. To the best of our knowledge, the present paper is the first to study regularity properties for local minimizers for nonlocal functionals of the form \eqref{eq:nonlocalfunctional} and local boundedness for weak solutions to \eqref{eq:nonlocalequation}.

In the following, we discuss related literature and describe the novelty of our results.

We first comment on related results for local operators. For this purpose, we consider functionals of the following form
\begin{equation*}
\int_{\Omega} f(x, u, \nabla u) \dx,
\end{equation*}
where $f$ is a non-negative function which describes the growth behavior of the functional. If the function $f$ satisfies the so-called $p$-growth condition, that is
\begin{equation*}
  |\xi|^p \lesssim f(x, z, \xi) \lesssim |\xi|^p+1 \qquad \text{ for } p>1,
\end{equation*}
the literature is very rich and many regularity results have been proved. We refer the reader to the classical references \cite{LaUr68,Fre75,GiGi82} and for a more comprehensive treatment to the books \cite{Gia83} and \cite{Giu03}.

Functionals with non-standard growth of $(p, q)$-type
\[  |\xi|^p \lesssim f(x, z, \xi) \lesssim |\xi|^q+1, \]
where $1<p<q$, are naturally connected to Orlicz--spaces. The analysis of regularity of minimizers of functionals having non-standard growth of $(p,q)$-type
was initiated by Marcellini's work \cite{Mar89}, where he studies strictly convex $C^2$-functions $f$ satisfying $(p,q)$-growth condition. \\
To the best of our knowledge, functionals with non-standard growth functions of the type \eqref{eq:pq} first appeared in the papers
\cite{Lie91} and \cite{MoNa91} in the context of regularity results.

In the paper \cite{Lie91}, Lieberman proves several regularity results for bounded weak solutions to a class of elliptic operators in divergence form. Furthermore, he studies quasiminimizers and proves regularity results for functions in corresponding De Giorgi classes.

In \cite{MoNa91}, Moscariello and Nania prove H\"older continuity of locally bounded minimizers for growth functions satisfying \eqref{eq:pq} and local boundedness for functions with $(p,q)$-growth. Their key idea is to introduce an auxiliary function which is comparable to the growth function $f$ and to prove that any function in the De Giorgi class corresponding to the auxiliary function is H\"older continuous. We adapt this idea for the proof of \Cref{thm:minimizer} and \Cref{thm:weaksol} to the nonlocal case, see \Cref{sec:hoelder}.

There have been many important contributions to regularity for problems with non-standard growth of $(p,q)$-type. Papers studying local operators with non-standard growth of $(p,q)$-type are, among others, \cite{Mar91, BhLe91, Mar93, MaPa94, Ci97, DaMa98, AcMi01, EsLeMi04, MaPa06, DiStVe09, BrStVe11, GiPa13, Ok17, Ok20, ZhTa20}.
For a more detailed picture on problems with non-standard growth, including double-phase problems, problems with variable exponents, and anisotropic problems, we refer the reader to the surveys \cite{Min06} and \cite{MiRa21}.

In the case of nonlocal operators, the energy functional is defined by \eqref{eq:nonlocalfunctional}.
Local regularity results for the fractional $p$-Laplacian, that is $f(t)=t^p$, were first established in the papers \cite{DiKuPa14, DiKuPa16} by Di Castro, Kuusi and Palatucci. Another important contribution to regularity is the work \cite{Coz17}, where he studies minimizers to nonlocal energy functionals plus a possibly discontinuous
potential. The nonlocal energy has $p$ growth for a class of symmetric kernels comparable to $(1-s)|x-y|^{-n-sp}$. Furthermore, he studies weak solutions to the Euler--Lagrange equation. He uses the nonlocality of the functional to prove an improved Caccioppoli inequality with an additional term, which disappears as the fractional order $s$ goes to one. We follow Cozzi's ideas at several points in the present paper and also make use of some auxiliary results proved in \cite{Coz17} such as an isoperimetric inequality.
For further results on the regularity of the fractional $p$-Laplacian, we refer the reader to \cite{BrLi17, BrLiSc17, No20} and the references therein.

Lately, the interest in the analysis of nonlocal problems with non-standard growth has increased. For instance, regularity results for nonlocal double phase equations and nonlocal equation with variable exponents are proved in \cite{DePa19, FaZh21, ScMe21,ByOkSo21}, respectively \cite{ChKi21,Ok21}. However, we would like to note that both, double phase equations and equations with variable exponents, do not fall into our setup.
See also \cite{GoDeSr20, GiKuSr21a, GiKuSr21b} for further regularity results concerning nonlocal operators with non-standard growth.

As far as we know, first regularity results for fractional order Orlicz--Sobolev spaces have been proved in \cite{BoSaVi20}
by Fern\'{a}ndez Bonder, Salort and Vivas. The authors establish regularity results for weak solutions to the Dirichlet problem for the fractional $g$-Laplacian. They prove interior and up to the boundary H\"older regularity to the corresponding Dirichlet problem. \\
See also \cite{MoSaVi21}, where qualitative properties of solutions such as a Liouville type theorem and symmetry results are proved. 

The present paper is substantially different from \cite{BoSaVi20}. On the one hand, we do not only study weak solutions but also local minimizers for the functional $\mathcal{I}_f$. The present paper also allows for $p>1$ and is not restricted to the case $p\geq 2$. Furthermore, we use a completely different approach via De Giorgi classes.

\subsection*{Notation} We write $c$ and $C$ for strictly positive constants whose exact values are not important and might change from line to line. Furthermore, we use the notations $c=c(\cdot)$ and $C=C(\cdot)$ if we want to highlight all quantities the constant depends on. 

\subsection*{Outline} The paper consists of seven sections and is organized as follows. In \Cref{subsec:aux}, we introduce an auxiliary growth function and prove several properties for convex functions satisfying \eqref{eq:pq}. Moreover, in \Cref{subsec:Funct} we recall some functional inequalities. The fractional De Giorgi classes with general convex functions with non-standard growth are introduced in \Cref{sec:degiorgiclasses}. Furthermore, we introduce fractional Orlicz--Sobolev spaces. In \Cref{sec:hoelder} we prove H\"older continuity and in \Cref{sec:locbdd} local boundedness for functions in fractional De Giorgi classes. Finally, in \Cref{sec:minimizer} resp. \Cref{sec:weaksoln}, we show that minimizers resp. weak solutions belong to the fractional De Giorgi classes and prove \Cref{thm:minimizer} and \Cref{thm:weaksol}.

\section{Preliminaries}\label{sec:preliminaries}

In this section, we study properties of the growth functions $f$ under consideration and collect some functional inequalities.

\subsection{Auxiliary results on growth functions}\label{subsec:aux}

Let us collect several results in order to illustrate the assumption \eqref{eq:pq}. The first two lemmas provide equivalent conditions for the upper and lower bounds in \eqref{eq:pq}, respectively.

\begin{lemma}
\label{lemma:upper}
Let $q \geq 1$ and $f : [0,\infty) \to [0,\infty)$ be a differentiable function. Then the following are equivalent:
\begin{enumerate} [(i)]
\item \eqref{eq:pq-upper},
\item $t \mapsto t^{-q}f(t)$ is decreasing,
\item $f(\lambda t) \le \lambda^q f(t)$ for all $\lambda \ge 1$,
\item $\lambda^q f(t) \le f(\lambda t)$ for all $\lambda \le 1$.
\end{enumerate}
\end{lemma}

\begin{proof}
(i) $\Leftrightarrow$ (ii) follows from the observation that $\frac{\d}{\d t}(t^{-q}f(t)) = t^{-q-1}(t f'(t) - q f(t))$. (i) $\Leftrightarrow$ (iii) follows from the observation that (iii) can be rewritten as
\begin{equation}
\label{eq:upperhelp}
\frac{f(\lambda t)}{(\lambda t)^q} \le \frac{f(t)}{t^q}, \quad\text{for all } \lambda \geq 1.
\end{equation}
Since also (iv) can be rewritten as \eqref{eq:upperhelp}, the equivalence (iii) $\Leftrightarrow$ (iv) is trivial.
\end{proof}

\begin{lemma}
\label{lemma:lower}
Let $p \geq 1$ and $f : [0,\infty) \to [0,\infty)$ be a differentiable function. Then the following are equivalent:
\begin{enumerate}[(i)]
\item \eqref{eq:pq-lower},
\item $t \mapsto t^{-p}f(t)$ is increasing,
\item $\lambda^p f(t) \le f(\lambda t)$ for all $\lambda \ge 1$,
\item $f(\lambda t) \le \lambda^p f(t)$ for all $\lambda \le 1$.
\end{enumerate}
\end{lemma}

\begin{proof}
The proof works exactly like the proof of \Cref{lemma:upper}.
\end{proof}

The following lemma provides a useful property of convex functions.

\begin{lemma}
\label{lem:conv}
Let $f : [0,\infty) \to [0,\infty)$ be convex and $f(0) = 0$. Then, the function $t \mapsto f(t)/t$ is increasing. If $f$ is differentiable, then $f$ satisfies \eqref{eq:pq-lower} with $p=1$.
\end{lemma}

\begin{proof}
The assertions follow from $f(\lambda t) = f(\lambda t + (1-\lambda)0) \leq \lambda f(t) + (1-\lambda) f(0)$ and \Cref{lemma:lower} with $p=1$.
\end{proof}

As a consequence we obtain some doubling-type inequalities for $f'$. These inequalities play an important role for the tail estimates in the upcoming regularity theory.

\begin{corollary}
Let $1 \leq p \leq q$ and let $f : [0,\infty) \to [0,\infty)$ be a function satisfying \eqref{eq:pq}. Then,
\begin{align}
\label{eq:derivativedoubling1}
\frac{p}{q} \lambda^{p-1} f'(t) \leq f'(\lambda t) \leq \frac{q}{p}\lambda^{q-1} f'(t) \quad\text{for all } \lambda \geq 1, \\
\label{eq:derivativedoubling2}
\frac{p}{q} \lambda^{q-1} f'(t) \leq f'(\lambda t) \leq \frac{q}{p}\lambda^{p-1} f'(t) \quad\text{for all } \lambda \leq 1,
\end{align}
and
\begin{equation} \label{eq:der-subadd}
\frac{1}{2}f'(t) + \frac{1}{2}f'(s) \le f'(t + s) \leq \frac{q}{p} 2^{q-1}(f'(t) + f'(s))
\end{equation}
for all $t,s \geq 0$.
\end{corollary}

\begin{proof}
For the second inequality in \eqref{eq:derivativedoubling1}, we compute using \eqref{eq:pq} and \Cref{lemma:upper}
\begin{equation*}
f'(\lambda t) \leq q \frac{f(\lambda t)}{\lambda t} \leq q \lambda^{q-1} \frac{f(t)}{t} \leq \frac{q}{p} \lambda^{q-1}f'(t).
\end{equation*}
The first inequality in \eqref{eq:derivativedoubling1} and \eqref{eq:derivativedoubling2} can be proved in the same way. The first estimate in \eqref{eq:der-subadd} is a direct consequence of monotonicity of $f'$. For the second estimate in \eqref{eq:der-subadd}, we may assume that $t \leq s$. Then, we obtain
\begin{equation*}
f'(t+ s) \leq q \frac{f(t+s)}{(t+s)^p}(t+s)^{p-1} \leq q \frac{f(2s)}{2s} \leq q 2^{q-1} \frac{f(s)}{s} \leq \frac{q}{p} 2^{q-1} (f'(t)+f'(s))
\end{equation*}
by using \eqref{eq:pq}, \Cref{lemma:upper} and \Cref{lemma:lower}.
\end{proof}

Another useful property of convex functions is the following:

\begin{lemma} \label{lem:g-inv}
Let $f : [0,\infty) \to [0,\infty)$ be convex and $f(0)=0$. Let $c > 1$ and assume that for some $t,s > 0$ it holds that $f(t) \le c f(s)$. Then $t \le cs$.
\end{lemma}

\begin{proof}
Let $t,s > 0$ be such that $f(t) \le c f(s)$. We assume that $t > cs$. Then by \Cref{lem:conv} and \Cref{lemma:lower} with $p=1$, we have
\begin{equation*}
\frac{f(s)}{s} \le \frac{f(cs)}{cs} \le \frac{f(t)}{t} \le \frac{cf(s)}{t} < \frac{cf(s)}{cs} = \frac{f(s)}{s}.
\end{equation*}
This is a contradiction, so it must hold that $t \le cs$, as desired.
\end{proof}

One of the key ideas of proving H\"older regularity in \cite{MoNa91} is to construct $F$, which is a convex increasing function satisfying some growth conditions and the comparability of $g(t) := F(t^p)$ and $f(t)$. These properties are important in our framework as well but we also need the comparability of the derivatives of these functions for the regularity estimates.

\begin{proposition} (c.f. \cite{Fio91})
\label{prop:F}
Let $1 \leq p \leq q$ and $f : [0,\infty) \to [0,\infty)$ be a convex, increasing, and differentiable function satisfying $f(0)=0$. Define $F, g : [0,\infty) \to [0,\infty)$ by
\begin{equation} \label{eq:Fg}
F(t) = \int_0^{t^{1/p}} \frac{f(s)}{s} \ds \quad\text{and}\quad g(t) = F(t^p).
\end{equation}
If $f$ satisfies \eqref{eq:pq}, then $F$ is a convex increasing function satisfying
\begin{equation}
\label{eq:Fpq}
F(t) \le tF'(t) \le \frac{q}{p} F(t)
\end{equation}
and $g$ is an increasing function satisfying
\begin{align}
\label{eq:Fcomp}
\frac{1}{q} f(t) &\le g(t) \le \frac{1}{p} f(t),\\
\label{eq:derFcomp}
\frac{1}{q} f'(t) &\le g'(t) \le \frac{1}{p} f'(t).
\end{align}
\end{proposition}

\begin{proof}
First of all, the function $F$ is well-defined by \Cref{lem:conv}. The functions $F$ and $g$ are increasing by definition. Moreover, $F$ is convex since
\begin{equation*}
F''(t) = \frac{1}{p^2} t^{\frac{1}{p}-2}f'(t^{\frac{1}{p}}) -\frac{1}{p}t^{-2}f(t^{\frac{1}{p}}) \ge \frac{1}{p} t^{-2}f(t^{\frac{1}{p}}) - \frac{1}{p}t^{-2}f(t^{\frac{1}{p}}) = 0
\end{equation*}
by \eqref{eq:pq-lower}. Thus, the first inequality in \eqref{eq:Fpq} follows from \Cref{lem:conv} and \Cref{lemma:lower}. The second inequality follows from
\begin{equation*}
F(\lambda t) = \int_{0}^{(\lambda t)^{1/p}} \frac{f(s)}{s} \ds = \int_{0}^{t^{1/p}} \frac{f(\lambda^{1/p}s)}{s} \ds \le \lambda^{q/p}\int_{0}^{t^{1/p}} \frac{f(s)}{s} \ds = \lambda^{q/p} F(t),
\end{equation*}
where we used \Cref{lemma:upper}.
By \eqref{eq:pq} we have
\begin{equation*}
p \frac{f(s)}{s} \le f'(s) \le q \frac{f(s)}{s},
\end{equation*}
and after integrating from $0$ to $t$ and using that $f(0) = 0$, we deduce \eqref{eq:Fcomp}. Finally, we compute
\begin{equation*}
p g'(t) = p\frac{f(t)}{t} \le f'(t) \le q\frac{f(t)}{t} = q g'(t),
\end{equation*}
using \eqref{eq:pq} from where \eqref{eq:derFcomp} follows.
\end{proof}

We close this subsection with two estimates for convex functions. \Cref{lemma:convexlemma} and \Cref{lemma:CaccHelpLemma} are generalizations of \cite[Lemma 4.1 and 4.2]{Coz17}.

\begin{lemma}
\label{lemma:convexlemma}
Let $f : [0,\infty) \to [0,\infty)$ be convex, differentiable and $f(0)=0$. Then, for any $\theta \in [0,1]$ and $a, b \geq 0$:
\begin{equation*}
f(a+b) -f(a) \ge \theta f'(a)b + (1-\theta)f(b). 
\end{equation*}
\end{lemma}

\begin{proof}
The result is clear for $\theta = 0$ by the superadditivity of convex functions with $f(0)=0$. For $\theta = 1$, we compute
\begin{equation*}
f(a+b) - f(a) = \int_a^{a+b} f'(\tau) \,\d \tau \ge f'(a)b,
\end{equation*}
where we used the fact that $t \mapsto f'(t)$ is increasing since $f$ is convex. The result for $\theta \in (0,1)$ follows by interpolation.
\end{proof}

\begin{lemma}
\label{lemma:CaccHelpLemma}
Let $f : [0,\infty) \to [0,\infty)$ be convex and differentiable. Then, for every $\mu \in [0,1]$ and $a, b \ge 0$:
\begin{equation*}
f(\vert \mu a-b\vert) - f(\vert a-b \vert) \le f'(b)a.
\end{equation*}
\end{lemma}

\begin{proof}
Let us first assume that $b \ge a$. Then
\begin{equation*}
f(\vert \mu a-b\vert) - f(\vert a-b \vert) = \int_{b-a}^{b - \mu a} f'(\tau) \,\d \tau \le f'(b-\mu a)a(1-\mu) \le  f'(b)a.
\end{equation*}
If on the other hand $\mu a \le b < a$, then
\begin{equation*}
f(\vert \mu a-b\vert) - f(\vert a-b \vert) = \int_{a-b}^{b - \mu a} f'(\tau) \, \d \tau \le f'(b-\mu a)(2b-(1+\mu)a) \le f'(b)a.
\end{equation*}
It remains to consider the case $b < \mu a$, but then $f(\vert \mu a-b \vert) - f(\vert a-b \vert) \le 0$, so there is nothing to prove.
\end{proof}

\subsection{Functional inequalities}\label{subsec:Funct}

In this section, we collect some well-known functional inequalities which are useful for the application of De Giorgi's methods for nonlocal operators. While the first three results are embeddings for fractional Sobolev spaces, the last proposition is a fractional isoperimetric inequality.

\begin{lemma} \cite[Lemma 4.6]{Coz17} \label{lem:embedding}
Let $0 < \tilde{\sigma} < \sigma < 1$ and $1 \leq \tilde{p} < p$. Let $\Omega' \subset \Omega \subset \Rd$ be two bounded measurable sets, then for any $u \in W^{\sigma, p}(\Omega)$
\begin{equation*}
\left( \int_{\Omega'} \int_{\Omega} \frac{|u(x)-u(y)|^{\tilde{p}}}{|x-y|^{d+\tilde{\sigma}\tilde{p}}} \dy \dx \right)^{1/\tilde{p}} \leq C |\Omega'|^{\frac{p-\tilde{p}}{p\tilde{p}}} \mathrm{diam}(\Omega)^{\sigma-\tilde{\sigma}} \left( \int_{\Omega'} \int_{\Omega} \frac{|u(x)-u(y)|^{p}}{|x-y|^{d+\sigma p}} \dy \dx \right)^{1/p},
\end{equation*}
where
\begin{equation*}
C = \left( \frac{d(p-\tilde{p})}{(\sigma-\tilde{\sigma})p\tilde{p}} |B_1| \right)^{\frac{p-\tilde{p}}{p\tilde{p}}}.
\end{equation*}
\end{lemma}

\begin{theorem} \cite[Corollary 4.9]{Coz17} \label{thm:homo-sobolev}
Let $0 < s_0 \leq s < 1$ and $p \geq 1$ be such that $sp < d$. Let $u \in W^{s, p}_{0}(B_R)$ and assume $u=0$ on a set $\Omega \subset B_R$ with $|\Omega| \geq \gamma|B_R|$ for some $\gamma \in (0,1]$. Then, there exists a constant $C > 0$, depending on $d$, $s_0$, $p$ and $\gamma$, such that
\begin{equation*}
\|u\|_{L^{p^{\ast}}(B_R)}^p \leq C \frac{1-s}{(d-sp)^{p-1}} \int_{B_R} \int_{B_R} \frac{|u(x)-u(y)|^p}{|x-y|^{d+sp}} \dy \dx.
\end{equation*}
\end{theorem}

\begin{theorem} \cite{DNPV12,BBM02,MaSh02} \label{thm:frac-sobolev}
Let $0 < s_0 \leq s < 1$ and $p \geq 1$ be such that $sp < d$. Let $u \in W^{s, p}(B_R)$. Then, there exists a constant $C > 0$, depending on $d$, $s_0$ and $p$, such that
\begin{equation*}
\|u\|_{L^{p^{\ast}}(B_R)}^p \leq C \frac{1-s}{(d-sp)^{p-1}} \int_{B_R} \int_{B_R} \frac{|u(x)-u(y)|^p}{|x-y|^{d+sp}} \dy \dx + C R^{-sp} \int_{B_R} \Vert u \Vert_{L^p(B_R)}^p.
\end{equation*}
\end{theorem}

\begin{proposition} \cite[Proposition 5.1]{Coz17} \label{prop:isoperimetric}
Let $p > 1$, $C_0 > 0$ and $\gamma, \gamma_0 \in (0,1)$. Then, there exist constants $\bar{s} \in (0,1)$ and $C > 0$, depending on $d$, $p$, $\gamma$, $\gamma_0$ and $C_0$, such that if $s \in [\bar{s},1)$ and if $u \in W^{s, p}(B_R)$ satisfies
\begin{equation*}
\begin{split}
&|B_R \cap \lbrace u \leq h \rbrace| \geq \gamma|B_R|, \quad |B_R \cap \lbrace u \geq k \rbrace| \geq \gamma_0 |B_R| \quad\text{and} \\
&\|u\|_{L^{p}(B_R)}^p + (1-s)R^{sp} [u]_{W^{s, p}(B_R)}^{p} \leq C_0 R^d(k-h)^p \quad\text{for}~ k > h,
\end{split}
\end{equation*}
then
\begin{equation*}
\begin{split}
&(k-h) \Big( |B_R \cap \lbrace u \leq h \rbrace| |B_R \cap \lbrace u \geq k \rbrace| \Big)^{\frac{d-1}{d}} \\
&\leq CR^{d-2+s}(1-s)^{1/p} [u]_{W^{s, p}(B_R)} |B_R \cap \lbrace h < u < k \rbrace|^{\frac{p-1}{p}}.
\end{split}
\end{equation*}
\end{proposition}

\section{De Giorgi classes}\label{sec:degiorgiclasses}

In this section, we introduce fractional order Orlicz--Sobolev spaces and define fractional De Giorgi classes governed by convex functions having non-standard growth.

Fractional order Orlicz--Sobolev spaces have been introduced in \cite{BoSa19} by Fern\'{a}ndez Bonder and Salort. The authors prove several properties of the spaces including that the fractional order Orlicz--Sobolev space approximates some Orlicz--Sobolev space as the fractional parameter goes to $1$.
For further results concerning fractional Orlicz--Sobolev spaces, we refer the reader to \cite{BaOu20, BaOuTa20, AlCiPi21b, AlCiPi21, DeFeSa21} and the references therein.

Let $f: [0, \infty) \to [0, \infty)$ be a convex increasing function satisfying \eqref{eq:pq-upper} and $f(0)=0$. Let $s \in (0, 1)$ and $\Omega \subset \Rd$ be open. We define the {\it Orlicz} and {\it Orlicz--Sobolev spaces} by
\begin{equation*}
\begin{split}
L^f(\Omega) &= \lbrace u: \Omega \to \R ~\text{measurable}: \Phi_{L^f(\Omega)}(u) < \infty \rbrace, \\
W^{s, f}(\Omega) &= \lbrace u \in L^f(\Omega): \Phi_{W^{s, f}(\Omega)}(u) < \infty \rbrace,\\
V^{s, f}(\Omega|\Rd) &= \lbrace u \in L^f(\Omega): \Phi_{V^{s, f}(\Omega)}(u) < \infty \rbrace,
\end{split}
\end{equation*}
where $\Phi_{L^f(\Omega)}$, $\Phi_{W^{s, f}(\Omega)}$ and $\Phi_{V^{s,f}(\Omega|\Rd)}$ are {\it modulars} defined by
\begin{equation*}
\begin{split}
\Phi_{L^f(\Omega)}(u) &= \int_{\Omega} f(|u(x)|) \dx, \\
\Phi_{W^{s,f}(\Omega)}(u) &= (1-s)\int_{\Omega} \int_{\Omega} f\left( \frac{|u(x)-u(y)|}{|x-y|^s} \right) \frac{\dy \dx}{\vert x-y \vert^{d}} ,\\
\Phi_{V^{s,f}(\Omega|\Rd)}(u) &= (1-s)\iint_{(\Omega^c \times \Omega^c)^c} f\left( \frac{|u(x)-u(y)|}{|x-y|^s} \right) \frac{\dy \dx}{\vert x-y \vert^{d}} .
\end{split}
\end{equation*}
Under more restrictive assumptions on $f$, $L^f(\Omega)$, $W^{s, f}(\Omega)$, and $V^{s, f}(\Omega|\Rd)$ are Banach spaces endowed with the norms 
\begin{equation*}
\begin{split}
\|u\|_{L^f(\Omega)} &= \inf \lbrace \lambda > 0: \Phi_{L^f(\Omega)}(u/\lambda) \leq 1 \rbrace, \\
\|u\|_{W^{s,f}(\Omega)} &= \|u\|_{L^f(\Omega)} + [u]_{W^{s,f}(\Omega)} = \|u\|_{L^f(\Omega)} + \inf \lbrace \lambda > 0: \Phi_{W^{s,f}(\Omega)}(u/\lambda) \leq 1 \rbrace,\\
\|u\|_{V^{s,f}(\Omega|\Rd)} &= \|u\|_{L^f(\Omega)} + [u]_{V^{s,f}(\Omega|\Rd)} = \|u\|_{L^f(\Omega)} + \inf \lbrace \lambda > 0: \Phi_{V^{s,f}(\Omega|\Rd)}(u/\lambda) \leq 1 \rbrace.
\end{split}
\end{equation*}

Let us next define nonlocal tails, which capture the behavior of functions $u\in V^{s, f}(\Omega|\Rd)$ at large scales. For this purpose, we consider a convex and differentiable function $f$ and a generalized inverse function of its derivative $f'$. There are several definitions of generalized inverse functions in the literature, see \cite{EmHo13}, \cite{deLaFo20}, but we use the following definition for a generalized inverse of $f'$:
\begin{equation} \label{eq:inverse}
(f')^{-1}(y) = \inf \lbrace t: f'(t) \geq y \rbrace.
\end{equation}
The advantage of this definition is that \eqref{eq:inverse} enjoys the following properties, which play an important role in the proof of regularity estimates.

\begin{proposition}
Let $f : [0, \infty) \to [0, \infty)$ be a convex and differentiable function. Then
\begin{align}
\label{eq:f-finv}
&(f' \circ (f')^{-1})(y) \geq y \quad \text{for all } y \geq 0,\\
\label{eq:finv-f}
&((f')^{-1} \circ f')(t) \leq t \quad \text{for all } t \geq 0.
\end{align}
\end{proposition}

\begin{proof}
To prove \eqref{eq:f-finv}, let $y \geq 0$ and $t = (f')^{-1}(y)$. Then, there exists a sequence $(t_n)_{n\geq 1}$ such that $t_n \geq t$, $t_n \to t$, and $f'(t_n) \geq y$. Since $f'$ is continuous by Darboux's theorem, we obtain $f'(t) \geq y$ by taking the limit $n \to \infty$. Assertion \eqref{eq:finv-f} is obvious by definition of $(f')^{-1}$.
\end{proof}

We define the nonlocal $f'$-Tail by
\begin{equation}\label{eq:tail}
\mathrm{Tail}_{f'}(u; x_0, R) = R^s (f')^{-1} \left( (1-s)R^s \int_{\Rd \setminus B_R(x_0)} f' \left( \frac{|u(y)|}{|y-x_0|^s} \right) \frac{\d y}{|y-x_0|^{d+s}} \right).
\end{equation}

Note that $\mathrm{Tail}_{f'}$ coincides with the standard tail considered in \cite{DiKuPa14,DiKuPa16,Coz17} in the special case $f(t) \asymp t^p$. Moreover, it is natural in the sense that the following scaling property is satisfied:
\begin{equation*}
\mathrm{Tail}_{f'}(u; x_0, R) = \mathrm{Tail}_{f'(\cdot/R^s)}(u(R\cdot); x_0/R, 1).
\end{equation*}

We claim that the nonlocal $f'$-Tail is well-defined for functions in the fractional Orlicz--Sobolev space $V^{s, f}(\Omega|\Rd)$. To this end, we define the {\it Legendre transform} $f^{\ast}: [0, \infty) \to [0, \infty)$ by $f^{\ast}(s) = \sup_{t \in [0, \infty)} (st-f(t))$. It is well known that
\begin{equation}
\label{eq:legendre}
(f^{\ast}(f'))(t) = f'(t) t-f(t)
\end{equation}
and that
\begin{equation}
\label{eq:Fenchel}
st \le f(t) + f^{\ast}(s), \quad \text{for all } t,s \ge 0.
\end{equation}
Inequality \eqref{eq:Fenchel} is called {\it Fenchel's inequality}.

\begin{proposition} \label{prop:tail-finite}
Let $q >1$, $s \in (0,1)$ and $\Omega \subset \Rd$ be open. Let $f: [0, \infty) \to [0, \infty)$ be a convex increasing function satisfying \eqref{eq:pq-upper} and $f(0) = 0$. If $u \in V^{s, f}(\Omega|\Rd)$ and $B_R(x_0) \subset \Omega$, then $\mathrm{Tail}_{f'}(u; x_0, R) < \infty$.
\end{proposition}

\begin{proof}
Let $u \in V^{s, f}(\Omega|\Rd)$ and $B_R(x_0) \subset \Omega$. We may assume that $x_0 = 0$. In order to prove finiteness of the tail, it is sufficient to show that 
\begin{equation}\label{eq:eqintegrandtail}
  \int_{\Rd \setminus B_R} f' \left( \frac{|u(y)|}{|y|^s} \right) \frac{\d y}{|y|^{d+s}} < \infty.
\end{equation}
Since $B_R \subset\Omega$ and $|x-y|\leq 2|y|$ for $x\in B_{R}$, $y\in  B_{R}^c$, we have
\begin{equation}\label{eq:Vtotail}
\begin{aligned}
\infty &> \Phi_{L^{f}(\Omega)}(u) + \Phi_{V^{s, f}(\Omega|\Rd)}(u) \\
 &\geq C\int_{B_R} f(|u(x)|) \dx +C (1-s)\iint_{B_R\times B_R^c} f\left( \frac{|u(x)-u(y)|}{|y|^s} \right) \frac{\d y \dx}{|y|^{d}}
 \end{aligned}
\end{equation}
by using \eqref{eq:pq-upper}. Our aim is to estimate the right-hand side of \eqref{eq:Vtotail} from below by the expression in \eqref{eq:eqintegrandtail}. 
Note that by \Cref{lem:conv}, \Cref{lemma:lower} and \eqref{eq:pq-upper}: $ f(|u(x)|)(|y|^{-s} \wedge |y|^{-sq} ) \ge f(|y|^{-s}|u(x)|)$. Therefore,
\begin{equation*}
 f(|u(x)|)(|y|^{-s} \wedge |y|^{-sq} )+f\left( \frac{|u(x)-u(y)|}{|y|^s} \right) \geq C f\left( \frac{|u(x)|+|u(x)-u(y)|}{|y|^s} \right) \geq C f\left( \frac{|u(y)|}{|y|^s}\right),
\end{equation*}
where we also used $f(t + s) \le 2^q(f(t) + f(s))$. This is a direct consequence of monotonicity and \eqref{eq:pq-upper}.
Since $\int_{B_R^c} |y|^{-d}(|y|^{-s} \wedge |y|^{-sq}) \dy < \infty$, we estimate the right-hand side of \eqref{eq:Vtotail} by
\begin{align*}
 & \iint_{B_R\times B_R^c} f(|u(x)|) \frac{\d y\dx}{|y|^{d} (|y|^{s} \vee |y|^{sq})} + (1-s)\iint_{B_R\times B_R^c} f\left( \frac{|u(x)-u(y)|}{|y|^s} \right) \frac{\d y \dx}{|y|^{d}} \\
 & \geq C \iint_{B_R\times B_R^c} f\left( \frac{|u(y)|}{|y|^s} \right) \frac{\d y \dx}{|y|^{d}} \geq C\int_{B_R^c} f\left( \frac{|u(y)|}{|y|^s} \right) \frac{ \dy}{|y|^{d}}.
\end{align*}

Altogether,
\begin{equation}\label{eq:Vtotail2}
\int_{B_R^c} f\left( \frac{|u(y)|}{|y|^s} \right) \frac{ \dy}{|y|^{d}} < \infty.
\end{equation}
Note that we have $(f^{\ast}(f'))(t) \le (q-1) f(t)$ by \eqref{eq:legendre} and \eqref{eq:pq-upper}. Therefore, we obtain
\begin{align*}
 \int_{B_R^c}  f'\left( \frac{|u(y)|}{|y|^s} \right) \frac{\d y}{|y|^{d+s}} &\leq \int_{B_R^c} f^{\ast}\left(f'\left( \frac{|u(y)|}{|y|^s} \right)\right) \frac{\d y}{|y|^{d}} + \int_{B_R^c} f\left( \frac{1}{|y|^s} \right) \frac{\d y}{|y|^{d}}\\
 &\leq C\int_{B_R^c} f\left( \frac{|u(y)|}{|y|^s} \right) \frac{\d y}{|y|^{d}} + C,
\end{align*}
where we used Fenchel's inequality \eqref{eq:Fenchel}. Combining the previous estimate with \eqref{eq:Vtotail2} finishes the proof.
\end{proof}

Having defined the fractional order Orlicz--Sobolev spaces, we are ready to introduce De Giorgi classes that are suitable to our setting.

\begin{definition} \label{def:DG}
Let $q > 1$, $c > 0$, $s \in (0,1)$, and let $\Omega$ be open. Let $f : [0,\infty) \mapsto [0,\infty)$ be convex and differentiable with $f(0)=0$, $f(1) = 1$. We say that $u \in G_+(\Omega; q, c, s, f)$ if $u \in V^{s,f}(\Omega|\Rd)$ and if for every $x_0 \in \Omega$, $0 < r < R \le d(x_0, \partial \Omega)$, $k \in \R$, it holds
\begin{equation}
\label{eq:Caccestimate}
\begin{split}
&\Phi_{W^{s, f}(B_r(x_0))}(w_+) + (1-s) \int_{B_{r}(x_0)}\int_{A_{k}^{-}} f'\left( \frac{w_{-}(y)}{\vert x-y \vert^s} \right) \frac{w_{+}(x)}{\vert x-y \vert^s} \frac{\dy \dx}{\vert x-y \vert^{d}} \\
&\le c\left(\frac{R}{R-r}\right)^{q} \Phi_{L^f(B_R(x_0))}\left( \frac{w_{+}}{R^s} \right) \\
&\quad +c(1-s) \left(\frac{R}{R-r} \right)^{d+sq} \Vert w_+ \Vert_{L^1(B_{R})} \int_{\Rd \setminus B_{r}(x_0)} f' \left( \frac{w_+(y)}{\vert y-x_0 \vert^s} \right) \frac{\dy}{\vert y-x_0 \vert^{d+s}},
\end{split}
\end{equation}
where $w_{\pm}(x) = (u(x)- k)_{\pm}$ and $A_{k}^{-} = \lbrace y \in \Rd: u(y) < k \rbrace$. We say that $u \in G_-(\Omega; q, c, s, f)$ if \eqref{eq:Caccestimate} holds true with $w_+$, $w_-$ and $A_{k}^{-}$ replaced by $w_-$, $w_+$ and $A_{k}^{+} = \lbrace y \in \Rd: u(y) > k \rbrace$, respectively. Moreover, we denote by $G(\Omega; q, c, s, f) = G_+(\Omega; q, c, s, f) \cap G_-(\Omega; q, c, s, f)$.
\end{definition}

The De Giorgi classes under consideration will contain minimizers of \eqref{eq:nonlocalfunctional} and weak solutions for \eqref{eq:nonlocaloperator} under suitable additional assumptions on $f$, see \Cref{thm:DG-minimizer} and \Cref{thm:CaccSol}.

The following proposition allows us to infer H\"older regularity of minimizers of $\mathcal{I}_f$ from regularity of functions in $G(\Omega; q, c, s, F(\cdot^p))$, where $F$ is as in \Cref{prop:F}.

\begin{proposition} \label{prop:DG-F}
Let $1 < p \leq q$, $s \in (0,1)$ and $\Omega \subset \R^d$ be open. Let $f : [0,\infty) \mapsto [0,\infty)$ be convex and increasing with \eqref{eq:pq}. Then for every $c_1 > 0$, there exists $c_2 = c_2(c_1, p, q) > 0$ such that $G_{\pm}(\Omega; q, c_1, s, f) \subset G_{\pm}(\Omega; q, c_2, s, F(\cdot^p))$.
\end{proposition}

\begin{proof}
The proof follows directly from \eqref{eq:Fcomp} and \eqref{eq:derFcomp}.
\end{proof}

\section{H\"older estimate} \label{sec:hoelder}

In this section, we prove H\"older estimates for functions in the De Giorgi class $G(\Omega; q, c, s, g)$, where $g$ is the function given by \eqref{eq:Fg}. Let us first prove a growth lemma for functions in $G_-(\Omega; q, c, s, g)$.

\begin{theorem} \label{thm:growth}
Let $1<p\leq q$, $c, H > 0$, $R>0$, $s_0\in(0,1)$ and assume $s \in [s_0,1)$. Let $f: [0, \infty) \to [0, \infty)$ be a convex increasing function satisfying \eqref{eq:pq}. Suppose that $B_{4R} = B_{4R}(x_0) \subset \Omega$. Let $u \in G_-(\Omega; q, c, s, g)$ satisfy $0 \leq u \leq 2H$ in $B_{4R}$ and
\begin{equation} \label{eq:growth-measure}
|B_{2R} \cap \lbrace u \geq H \rbrace| \geq \gamma |B_{2R}|
\end{equation}
for some $\gamma \in (0,1)$. There exists $\delta \in (0,1)$ such that if
\begin{equation} \label{eq:growth-tail}
\mathrm{Tail}_{f'}(u_-; x_0, 4R) \leq \delta H,
\end{equation}
then
\begin{equation} \label{eq:growth}
u \geq \delta H \quad\text{in}~ B_{R}.
\end{equation}
The constant $\delta$ depends only on $d$, $s_0$, $p$, $q$, $c$ and $\gamma$.
\end{theorem}

\begin{proof}
Within the proof we use $C>0$ to denote a constant depending on $d$, $s_0$, $p$, $q$, $c$ and $\gamma$ and whose value might change from line to line. We may assume that $x_0 = 0$.

Let us assume
\begin{equation} \label{eq:growth-measure0}
|B_{2R} \cap \lbrace u < 2\delta H \rbrace| \leq \gamma_0 |B_{2R}|
\end{equation}
for some $\gamma_0 \in (0, 2^{-d-1}]$. We first prove the assertion of the lemma under the assumption \eqref{eq:growth-measure0} and then verify \eqref{eq:growth-measure0} using \eqref{eq:growth-measure}.

Let $0 < \tilde{\sigma}:=\max\lbrace s_0/4, 2s-1 \rbrace < \sigma:= \max\lbrace s_0/2, (3s-1)/2 \rbrace < s$. Then, we have
\begin{equation} \label{eq:sigma}
1-\tilde{\sigma} \leq 2(1-s), \quad 1-\sigma \leq \frac{3}{2}(1-s) \quad\text{and}\quad \sigma-\tilde{\sigma} \geq C(1-s)
\end{equation}
for some $C = C(s_0) > 0$. Indeed, for the last inequality in \eqref{eq:sigma}, we observe that
\begin{equation*}
\begin{split}
s_0/4 \geq 2s-1 &\implies \sigma-\tilde{\sigma} \geq \frac{s_0}{4} \geq \frac{s_0}{4(1-s_0)} (1-s) \quad\text{and}\\
s_0/4 < 2s-1 &\implies \sigma-\tilde{\sigma} \geq \frac{3s-1}{2} - (2s-1) = \frac{1}{2}(1-s).
\end{split}
\end{equation*}

Let $\delta \in (0, 1/8)$ to be determined later. Let $\delta H \leq h < k \leq H$, $R \leq \rho < \tau \leq 2R$ and define $w_\pm = (u-k)_\pm$, $A^-_{k,R} = \{ x \in B_R : u(x) < k \} = \supp(w_-) \cap B_R$. By \eqref{eq:growth-measure0} we have
\begin{equation*}
\begin{split}
|B_{\rho} \cap \lbrace w_- = 0 \rbrace|
&\geq |B_{\rho}| - |B_{2R} \cap\{u<2\delta H\}| \geq |B_{\rho}| - \gamma_0 |B_{2R}| \\
&= \left(1-\gamma_0 \left(\frac{2R}{\rho}\right)^d\right)|B_{\rho}| \geq \frac{1}{2}|B_{\rho}|.
\end{split}
\end{equation*}
Thus, we apply \Cref{thm:homo-sobolev} to $w$ to obtain
\begin{equation*}
(k-h) |A_{h, \rho}^-|^{\frac{d-\tilde{\sigma}}{d}} \leq \left(\int_{B_{\rho}} w_-(x)^{\frac{d}{d-\tilde{\sigma}}} \, \d x  \right)^{\frac{d-\tilde{\sigma}}{d}} \leq C(1-\tilde{\sigma}) \int_{B_{\rho}}\int_{B_{\rho}} \frac{|w_-(x) - w_-(y)|}{|x-y|^{d+\tilde{\sigma}}} \dy \dx.
\end{equation*}
Moreover, by \eqref{eq:sigma} and \Cref{lem:embedding} we have
\begin{equation*}
\begin{split}
(k-h) |A_{h, \rho}^-|^{\frac{d-\tilde{\sigma}}{d}} 
&\leq C(1-s) \int_{A_{k, \rho}^-}\int_{B_{\rho}} \frac{|w_-(x)-w_-(y)|}{|x-y|^{d+\tilde{\sigma}}} \dy \dx \\
&\leq C \rho^{\sigma-\tilde{\sigma}} |A_{k, \tau}^-|^{\frac{p-1}{p}} \left((1-s)\int_{B_\rho} \int_{B_\rho} \frac{|w_-(x)-w_-(y)|^p}{|x-y|^{d+\sigma p}} \dy \dx \right)^{1/p},
\end{split}
\end{equation*}
or equivalently,
\begin{equation*}
\frac{(k-h)^p |A_{h, \rho}^-|^{\frac{d-\tilde{\sigma}}{d}p}}{C|A_{k,\tau}^-|^{p-1} \rho^{(\sigma-\tilde{\sigma})p} \mu(B_\rho \times B_\rho)} \leq \fint_{B_\rho \times B_\rho} \frac{|w_-(x) - w_-(y)|^p}{|x-y|^{sp}} \mu(\d X),
\end{equation*}
where $\mu(\d X) = (1-s) |x-y|^{-d+(s-\sigma)p} \dy \dx$. Since $F$ is increasing and convex, Jensen's inequality yields
\begin{equation} \label{eq:applyF}
\begin{split}
F\left( \frac{(k-h)^p |A_{h, \rho}^-|^{\frac{d-\tilde{\sigma}}{d}p}}{C|A_{k,\tau}^-|^{p-1} \rho^{(\sigma-\tilde{\sigma})p} \mu(B_\rho \times B_\rho)} \right)
&\leq \fint_{B_\rho \times B_\rho} F\left( \frac{|w_-(x)-w_-(y)|^p}{|x-y|^{sp}} \right) \mu(\d X) \\
&= \frac{C \rho^{(s-\sigma)p}}{\mu(B_\rho \times B_\rho)} \Phi_{W^{s, g}(B_\rho)}(w_-).
\end{split}
\end{equation}

By definition of $G_-(\Omega; q, c, s, g)$, we have
\begin{equation*}
\begin{split}
&\Phi_{W^{s, g}(B_\rho)}(w_-) + (1-s) \int_{B_\rho} \int_{A_k^+} g' \left( \frac{w_+(y)}{|x-y|^s} \right) \frac{w_-(x)}{|x-y|^s} \frac{\dy \dx}{\vert x-y \vert^{d}}  \\
&\leq c \left(\frac{\tau}{\tau-\rho}\right)^{q} \Phi_{L^{g}(B_\tau)} \left( \frac{w_-}{\tau^s} \right) + c(1-s) \left(\frac{\tau}{\tau - \rho}\right)^{d+sq} \Vert w_- \Vert_{L^1(B_{\tau})} \int_{\Rd \setminus B_\rho} g'\left( \frac{w_-(y)}{|y|^s} \right) \frac{\d y}{|y|^{d+s}},
\end{split}
\end{equation*}
where $A_{k}^{+} = \lbrace y \in \Rd: u(y) < k \rbrace$. Let us estimate the right-hand side of this inequality. Using the assumptions that $u \geq 0$ in $B_{4R}$ and the fact that $F$ is increasing, we estimate $\|w_-\|_{L^1(B_\tau)} \leq C |A_{k,\tau}^-| k$ and
\begin{equation*}
\Phi_{L^{g}(B_\tau)} \left( \frac{w_-}{\tau^s} \right) \leq C|A_{k,\tau}^-| F\left(\left( \frac{k}{R^s}\right)^p \right).
\end{equation*}
Moreover, using \eqref{eq:derFcomp} and \eqref{eq:der-subadd}, we obtain
\begin{equation*}
(1-s)\int_{\Rd \setminus B_\rho} g'\left( \frac{w_-(y)}{|y|^s} \right) \frac{\d y}{|y|^{d+s}} \leq C(1-s) \int_{\Rd \setminus B_{\rho}} \left( f'\left(\frac{k}{R^s}\right) + f'\left(\frac{u_-(y)}{\vert y \vert^s}\right) \right) \frac{\d y}{|y|^{d+s}}.
\end{equation*}
The first term is controlled by $CR^{-s} f'(k/R^s)$. For the second term, we use \eqref{eq:growth-tail}, \eqref{eq:f-finv} and $\delta H < k$ to obtain
\begin{equation} \label{eq:assumption-g}
C(1-s) \int_{\Rd \setminus B_\rho} f' \left( \frac{u_-(y)}{|y|^s} \right) \frac{\d y}{|y|^{d+s}} \leq \frac{C}{R^s} f' \left( \frac{\delta H}{R^s} \right) \leq \frac{C}{R^s} f' \left( \frac{k}{R^s} \right).
\end{equation}
Furthermore, we have
\begin{equation*}
\frac{1}{R^s} f'\left( \frac{k}{R^s} \right) \leq \frac{C}{k}f\left( \frac{k}{R^s} \right) \le \frac{C}{k} F\left( \left( \frac{k}{R^s} \right)^p \right)
\end{equation*}
by \eqref{eq:pq-upper} and \eqref{eq:Fcomp}. Therefore, we have estimated
\begin{equation}  \label{eq:Caccio-F}
\begin{split}
\Phi_{W^{s, g}(B_\rho)}(w_-) + (1-s) \int_{B_\rho} \int_{A_{k}^{+}} g' \left( \frac{w_+(y)}{|x-y|^s} \right) \frac{w_-(x)}{|x-y|^s} \frac{\dy \dx}{\vert x-y \vert^{d}}  \\
\leq C \left(\frac{\tau}{\tau-\rho}\right)^{d+q} |A_{k,\tau}^-| F \left( \left( \frac{k}{R^s} \right)^p \right).
\end{split}
\end{equation}

Combining \eqref{eq:applyF} and \eqref{eq:Caccio-F}, we can find a constant $C>0$, depending on $d$, $s_0$, $p$, $q$, $c$ and $\gamma$, such that
\begin{equation*}
F\left( \frac{(k-h)^p |A_{h, \rho}^-|^{\frac{d-\tilde{\sigma}}{d}p}}{C|A_{k,\tau}^-|^{p-1} \rho^{(\sigma-\tilde{\sigma})p} \mu(B_\rho \times B_\rho)} \right) \leq \frac{C \rho^{(s-\sigma)p} |A_{k,\tau}^-|}{\mu(B_\rho \times B_\rho)}\left(\frac{\tau}{\tau-\rho}\right)^{d+q} F \left( \left( \frac{k}{R^s} \right)^p \right).
\end{equation*}
Using \Cref{lem:g-inv}, we deduce that
\begin{equation*}
(k-h) |A_{h, \rho}^-|^{\frac{d-\tilde{\sigma}}{d}} \leq C \left(\frac{\tau}{\tau-\rho}\right)^{\frac{d+q}{p}} \frac{k}{R^{\tilde{\sigma}}} |A_{k,\tau}^-|.
\end{equation*}
We iterate this inequality with $k = k_j$, $h = k_{j+1}$, $\rho = R_{j+1}$, and $\tau = R_j$, where
\begin{equation*}
R_j = (1+2^{-j})R \quad\text{and}\quad k_j = (1+2^{-j}) \delta H
\end{equation*}
for $j \in \mathbb{N} \cup \lbrace 0 \rbrace$. Let $y_j = |A_{k_{j}, R_{j}}^{-}| / |B_{R_j}|$, then
\begin{equation*}
\frac{\delta H}{2^{j+1}} \left( y_{j+1} |B_{R_{j+1}}| \right)^{\frac{d-\tilde{\sigma}}{d}} \leq C 2^{\frac{d+q}{p} j} \frac{\delta H}{R^{\tilde{\sigma}}} y_j |B_{R_j}|.
\end{equation*}
In other words, we have $y_{j+1} \leq C b^j y_j^{d/(d-\tilde{\sigma})} \leq C b^j \max \lbrace y_j^{1+\beta_1}, y_j^{1+\beta_2} \rbrace$, where
\begin{equation*}
\beta_1 = \frac{1}{d-1}, \quad \beta_2 = \frac{s_0/4}{d-s_0/4}, \quad\text{and}\quad b = 2^{(\frac{d+q}{p}+1)\frac{d}{d-1}}.
\end{equation*}
Thus, $y_j \to 0$ as $j \to \infty$, provided that $y_0 \leq \min\lbrace C^{-1/\beta_2} b^{-1/\beta_2^2}, C^{-1/\beta_1} b^{-1/\beta_1^2} \rbrace$. See \cite[Lemma 4.4]{ChKi21}. By taking $\gamma_0 \leq \min\lbrace C^{-1/\beta_2} b^{-1/\beta_2^2}, C^{-1/\beta_1}  b^{-1/\beta_1^2}\rbrace$, we conclude \eqref{eq:growth} from \eqref{eq:growth-measure0}.

Let us next prove \eqref{eq:growth-measure0} by contradiction using the assumption \eqref{eq:growth-measure}. Suppose that \eqref{eq:growth-measure0} does not hold, i.e.,
\begin{equation} \label{eq:contrad}
|B_{2R} \cap \lbrace u < 2\delta H \rbrace| > \gamma_0 |B_{2R}|.
\end{equation}
Let $\bar{s}$ be the constant given in \Cref{prop:isoperimetric}. We distinguish two cases $s \in [\bar{s}, 1)$ and $s \in (0, \bar{s})$. For the first case, we let $l$ be the unique integer such that $2^{-l-1} \leq \delta < 2^{-l}$ and set $k_i = 2^{-i}H$ for $i=0, 1, \dots, l-1$. To apply \Cref{prop:isoperimetric} to $(u-k_{i-1})_-$ with $h = k_{i-1} - k_i$ and $k = k_{i-1} - k_{i+1}$, we check the following conditions: By \eqref{eq:growth-measure} and \eqref{eq:contrad}
\begin{equation} \label{eq:levelset-h}
|B_{2R} \cap \lbrace (u-k_{i-1})_- \leq h \rbrace| = |B_{2R} \cap \lbrace u \geq k_i \rbrace| \geq |B_{2R} \cap \lbrace u \geq H \rbrace| \geq \gamma |B_{2R}|
\end{equation}
and
\begin{equation*}
|B_{2R} \cap \lbrace (u-k_{i-1})_- \geq k \rbrace| = |B_{2R} \cap \lbrace u \leq k_{i+1} \rbrace| \geq |B_{2R} \cap \lbrace u < 2\delta H \rbrace| > \gamma_0 |B_{2R}|
\end{equation*}
for $i=1, \dots, l-2$. 
Moreover, we prove that there is a constant $C>0$ such that
\begin{equation*}
\|(u-k_{i-1})_-\|_{L^p(B_R)}^p + (1-\sigma) R^{\sigma p} [(u-k_{i-1})_-]_{W^{\sigma, p}(B_R)}^p \leq CR^d (k_i-k_{i+1})^p.
\end{equation*}
Indeed, it follows from $u \geq 0$ in $B_{4R}$ that
\begin{equation*}
\|(u-k_{i-1})_-\|_{L^p(B_R)}^p \leq C R^d k_{i-1}^p.
\end{equation*}
The estimate
\begin{equation} \label{eq:seminorm-sigma-p}
(1-\sigma)R^{\sigma p} [(u-k_{i-1})_-]_{W^{\sigma, p}(B_R)}^p \leq CR^d k_{i-1}^p
\end{equation}
follows from \eqref{eq:sigma} and the computation
\begin{equation}
\begin{split}
&F\left(\fint_{B_R \times B_R} \frac{\vert (u(x) - k_{i-1})_- - (u(y)-k_{i-1})_- \vert^p}{\vert x-y \vert^{sp}} \d \mu(X) \right)\\
&\le \frac{R^{(s-\sigma)p}}{\mu(B_R \times B_R)} (1-s) \int_{B_R} \int_{B_R} F\left( \frac{\vert (u(x) - k_{i-1})_- - (u(y)-k_{i-1})_- \vert^p}{\vert x-y \vert^{sp}} \right) \frac{\d y \dx}{\vert x-y \vert^{d}} \\
&\le C\frac{R^{(s-\sigma)p}}{\mu(B_R \times B_R)}\vert B_{2R} \cap \{ u \le k_{i-1} \}\vert F\left(\left( \frac{k_{i-1}}{R^s} \right)^p \right),
\end{split}
\end{equation}
which can be obtained along the lines of the first part of this proof. Estimating $\vert B_{2R} \cap \{ u \le k_{i-1} \}\vert \le CR^d$ and applying \Cref{lem:g-inv}, we obtain \eqref{eq:seminorm-sigma-p}, as desired.

Therefore, by applying \Cref{prop:isoperimetric} and using \eqref{eq:levelset-h} we have
\begin{equation*}
(k_i - k_{i+1})^p |B_{2R} \cap \lbrace u \leq k_{i+1} \rbrace|^{\frac{d-1}{d}p} \leq C(1-\sigma) R^{(-1+\sigma)p} [(u-k_{i-1})_-]^p_{W^{\sigma, p}(B_{2R})} |D_i|^{p-1},
\end{equation*}
where $D_i = B_{2R} \cap \lbrace h \leq (u-k_{i-1})_- < k \rbrace = B_{2R} \cap \lbrace k_{i+1} < u \leq k_i \rbrace$. Using \eqref{eq:seminorm-sigma-p} and $k_{i+1} \geq 2\delta H$, we obtain
\begin{equation*}
|B_{2R} \cap \lbrace u \leq 2\delta H \rbrace|^{\frac{d-1}{d} \frac{p}{p-1}} \leq C R^{\frac{d-p}{p-1}} |D_i|.
\end{equation*}
We sum up this inequality over $i=1, \dots, l-2$ to derive
\begin{equation*}
(l-2)|B_{2R} \cap \lbrace u \leq 2\delta H \rbrace|^{\frac{d-1}{d} \frac{p}{p-1}} \leq C R^{\frac{d-p}{p-1}} R^d,
\end{equation*}
from which we conclude by definition of $l$
\begin{equation*}
|B_{2R} \cap \lbrace u \leq  2\delta H \rbrace| \leq C |B_{2R}| |\log \delta|^{-\frac{d}{d-1} \frac{p-1}{p}}.
\end{equation*}
Therefore, we arrive at a contradiction by using \eqref{eq:contrad} and taking $\delta$ sufficiently small. The proof for the case $s \in [\bar{s}, 1)$ is finished. 

For the case $s \in (0, \bar{s})$, we use the estimate \eqref{eq:Caccio-F} with $k=4\delta H$ to obtain
\begin{equation*}
\begin{split}
&CR^d F \left( \left( \frac{4\delta H}{R^s} \right)^p \right) \geq (1-\bar{s}) \int_{B_{2R}} \int_{A_{4\delta H, 2R}^{+}} g'\left( \frac{w_+(y)}{|x-y|^s} \right) \frac{w_-(x)}{|x-y|^s} \frac{\dy \dx}{|x-y|^d} \\
&\geq \frac{C}{R^{d+s}} \left( \int_{B_{2R} \cap \lbrace u \geq H \rbrace} g'\left( \frac{u(y)-4\delta H}{(2R)^s} \right) \dy \right) \left( \int_{B_{2R} \cap \lbrace u < 2\delta H \rbrace} (4\delta H - u(x)) \dx \right) \\
&\geq \frac{C}{R^{d}} 2\delta \frac{H}{4R^s} g'\left( \frac{H}{4R^s} \right) |B_{2R} \cap \lbrace u \geq H \rbrace| |B_{2R} \cap \lbrace u < 2\delta H \rbrace|.
\end{split}
\end{equation*}
By \eqref{eq:growth-measure} and \eqref{eq:contrad}, we obtain
\begin{equation*}
\delta F\left( \left( \frac{H}{4R^s} \right)^p \right) \leq C F \left( \left( \frac{4\delta H}{R^s} \right)^p \right),
\end{equation*}
where we also used \eqref{eq:derFcomp}, \eqref{eq:pq-lower} and \eqref{eq:Fcomp}. Therefore, by \Cref{lem:g-inv} we obtain $\delta \leq C \delta^p$.  Since $p > 1$, we arrive at a contradiction by taking $\delta$ sufficiently small.
\end{proof}

Using \Cref{thm:growth}, we prove H\"older estimates for functions in $G(\Omega; q, c, s, g)$.

\begin{theorem}
\label{thm:Fholder}
Let $1<p\leq q$, $c > 0$, $s_0 \in (0,1)$, and assume $s \in [s_0,1)$. Let $f: [0, \infty) \to [0, \infty)$ be a convex increasing function satisfying \eqref{eq:pq}. Then, there exist $\alpha \in (0,1)$ and $C > 0$, depending on $d$, $s_0$, $p$, $q$ and $c$, such that for every $u \in G(\Omega; q, c, s, g)$ and any $B_{8R}(x_0) \subset \Omega$,
\begin{equation*}
R^\alpha [u]_{C^{\alpha}(\overline{B_R(x_0)})} \leq C \|u\|_{L^{\infty}(B_{4R}(x_0))} + \mathrm{Tail}_{f'}(u; x_0, 4R).
\end{equation*}
\end{theorem}

\begin{proof}
Let $B_{8R}(x_0) \subset \Omega$. We may assume that $x_0 = 0$ and that $\|u\|_{L^{\infty}(B_{4R})} < \infty$. The idea of the proof is to find a small constant $\alpha \in (0,1)$ and to construct a non-increasing sequence $(M_j)$ and a non-decreasing sequence $(m_j)$ satisfying
\begin{equation} \label{eq:sequences}
m_j \leq u \leq M_j \quad \text{in}~ B_{4R_j} \quad\text{and}\quad M_j - m_j = L 4^{-\alpha j},
\end{equation}
for all $j \geq 0$, where $R_j = 4^{-j}R$ and
\begin{equation*}
L = C_0 \|u\|_{L^{\infty}(B_{4R})} + \mathrm{Tail}_{f'}(u; 0, 4R)
\end{equation*}
for some $C_0 > 0$. Once we construct such sequences, the desired result follows by a standard argument.

We set $M_j = 4^{-\alpha j} L/2$ and $m_j = -4^{-\alpha j}L/2$ for $j=0, 1, \dots, j_0$ for some $j_0 \in \mathbb{N}$ to be determined later. Moreover, we take $C_0$ sufficiently large so that $C_0 \geq 2 \cdot 4^{\alpha j_0}$. This ensures that $M_j$ and $m_j$ satisfy \eqref{eq:sequences} up to $j_0$. Let us now fix $j \geq j_0$ and suppose that the sequences $(M_j)$ and $(m_j)$ have been constructed up to $j$. It is enough to construct $M_{j+1}$ and $m_{j+1}$ satisfying \eqref{eq:sequences}. We first assume
\begin{equation}\label{eq:Hoeldcase1}
|B_{2R_j} \cap \lbrace u \geq m_{j} + (M_j-m_j)/2 \rbrace| \geq \frac{1}{2} |B_{2R_j}|. 
\end{equation}
In this case, we define $v=u-m_j$ and set $H=(M_j-m_j)/2 = 4^{-\alpha j}L/2$. Then, $0 \leq v \leq 2H$ in $B_{4R_j}$ and
\begin{equation*}
|B_{2R_j} \cap \lbrace v \geq H \rbrace| \geq \frac{1}{2} |B_{2R_j}|.
\end{equation*}
To apply \Cref{thm:growth} to $v$, we let $\delta$ be the constant in \Cref{thm:growth} and verify \eqref{eq:growth-tail}. Indeed, it is easy to see that
\begin{equation*}
v(y) \geq -2H \left( \left( \frac{|y|}{R_j} \right)^\alpha -1 \right)
\end{equation*}
for $y \in B_{4R} \setminus B_{4R_j}$ and $v(y) \geq -|u(y)| - L/2$ for $y \in \Rd \setminus B_{4R}$. Thus, using \eqref{eq:der-subadd}
\begin{equation*}
\begin{split}
&\int_{\Rd \setminus B_{4R_j}} f'\left( \frac{v_-(y)}{|y|^s} \right) |y|^{-d-s} \dy \\
&\leq C\int_{B_{4R} \setminus B_{4R_j}} f'\left( \frac{2H((|y|/R_j)^\alpha-1)}{|y|^s} \right) \frac{\dy}{|y|^{d+s}} + C\int_{\Rd \setminus B_{4R}} f'\left( \frac{|u(y)|+L/2}{|y|^s} \right) \frac{\dy}{|y|^{d+s}} \\
&=: J_1 + J_2.
\end{split}
\end{equation*}
Using the change of variables, we obtain
\begin{equation*}
J_1 \leq \frac{C}{R_{j}^{s}} \int_{\Rd \setminus B_4} f'\left( \frac{2H(|y|^\alpha-1)}{R_j^s |y|^s} \right) \frac{\d y}{|y|^{d+s}} \leq \frac{C}{R_{j}^{s}} \int_4^{\infty} f'\left( \frac{2H(\rho^\alpha-1)}{R_j^s \rho^s} \right) \frac{\d \rho}{\rho^{1+s}}.
\end{equation*}
By \eqref{eq:derivativedoubling1} and \eqref{eq:derivativedoubling2} we have
\begin{equation*}
f'\left( \frac{2H(\rho^\alpha-1)}{R_j^s \rho^s} \right) \leq \frac{q}{p} \max \left \lbrace \left( \frac{8(\rho^\alpha-1)}{\delta \rho^s} \right)^{q-1}, \left( \frac{8(\rho^\alpha-1)}{\delta \rho^s} \right)^{p-1} \right \rbrace f'\left( \frac{\delta H}{(4R_j)^s} \right),
\end{equation*}
and hence
\begin{equation*}
J_1 \leq \frac{C}{(4R_j)^s} \left( \int_{4}^{\infty} \frac{(\rho^\alpha-1)^{q-1}}{\rho^{1+s_0p}} \, \d\rho \right) f'\left(\frac{\delta H}{(4R_j)^s} \right).
\end{equation*}
Taking $\alpha = \alpha(d, s_0, p, q) \in (0,1)$ sufficiently small so that
\begin{equation}
\label{eq:alphacond}
C \int_4^{\infty} \frac{(\rho^\alpha-1)^{q-1}}{\rho^{1+s_0p}} \,\d\rho \leq \frac{1}{2},
\end{equation}
we obtain
\begin{equation*}
J_1 \leq \frac{1}{2} (4R_{j})^{-s} f'\left( \frac{\delta H}{(4R_j)^s} \right).
\end{equation*}

For $J_2$, we use \eqref{eq:der-subadd} to deduce
\begin{equation*}
J_2 \leq C \left( \int_{\Rd \setminus B_{4R}} f'\left( \frac{|u(y)|}{|y|^s} \right) \frac{\dy}{|y|^{d+s}} + \int_{\Rd \setminus B_{4R}} f'\left( \frac{L/2}{|y|^s} \right) \frac{\dy}{|y|^{d+s}} \right).
\end{equation*}
Since $L \geq \mathrm{Tail}_{f'}(u; 0, 4R)$, it follows from \eqref{eq:f-finv} and the definition of $H$
\begin{equation*}
\int_{\Rd \setminus B_{4R}} f'\left( \frac{|u(y)|}{|y|^s} \right) \frac{\dy}{|y|^{d+s}}
\leq R^{-s} f'\left( \frac{L}{R^s} \right) = R^{-s} f'\left( \frac{2H4^{\alpha j}}{R^s} \right).
\end{equation*}
Choosing $j_0$ sufficiently large so that $8\cdot 4^{(\alpha-s_0)j_0} \leq \widetilde{\delta}$, for some $\widetilde{\delta} < \delta$ to be determined later, we have
\begin{equation*}
R^{-s} f'\left( \frac{2H4^{\alpha j}}{R^s} \right) \leq (4R_j)^{-s} f'\left( \frac{\widetilde{\delta} H}{(4R_j)^s} \right).
\end{equation*}
Similarly, by \eqref{eq:derivativedoubling2}
\begin{equation*}
\begin{split}
\int_{\Rd \setminus B_{4R}} f'\left( \frac{L/2}{|y|^s} \right) \frac{\dy}{|y|^{d+s}}
&= \frac{1}{R^s} \int_{\Rd \setminus B_{4}} f'\left( \frac{L/2}{R^s |y|^s} \right) \frac{\dy}{|y|^{d+s}} \\
&\leq C_1 \left( \int_{\Rd \setminus B_{4}} |y|^{-s(p-1)} \frac{\dy}{|y|^{d+s}} \right) \frac{1}{R^s} f'\left( \frac{L}{R^s} \right) \\
&\leq C_2(4R_j)^{-s} f'\left( \frac{\widetilde{\delta} H}{(4R_j)^s} \right)
\end{split}
\end{equation*}
for some $C_1,C_2 \ge 1$ depending on $d$, $s_0$ and $p$.
We now choose $\widetilde{\delta} = (2^{q+1} C_2)^{-\frac{1}{p-1}}\delta > 0$, and obtain:
\begin{equation*}
(1-s)\int_{\Rd \setminus B_{4R_j}} f'\left( \frac{v_-(y)}{|y|^s} \right) |y|^{-d-s} \dy \leq J_1 + J_2 \leq (4R_j)^{-s} f'\left( \frac{\delta H}{(4R_j)^s} \right).
\end{equation*}
This inequality together with \eqref{eq:finv-f} verify \eqref{eq:growth-tail} and allow us to apply \Cref{thm:growth} to $v$. Therefore, we obtain $v \geq \delta H$ in $B_{R_j}$, which implies
\begin{equation*}
u \geq  m_j + \delta H = m_j + 4^{-\alpha j} \frac{\delta}{2}L  \ge  m_j + 4^{-\alpha j}(1-4^{-\alpha})L \quad\text{in}~ B_{R_j},
\end{equation*}
upon choosing $\alpha \in (0,1)$ so small that it satisfies \eqref{eq:alphacond} and $\alpha < \log_4\left( \frac{2}{2-\delta} \right)$.

Therefore, we define $M_{j+1} = M_j$ and $m_{j+1} = m_j + 4^{-\alpha j} (1-4^{-\alpha})L$ in the case \eqref{eq:Hoeldcase1}. The other case can be proved in a similar way.
\end{proof}

\section{Local boundedness}\label{sec:locbdd}

The goal of this section is to prove local boundedness of functions $u \in G(\Omega; q, c, s, f)$. More precisely, we prove that a function $u \in G_+(\Omega; q, c, s, f)$ is locally bounded from above. Similarly, one can prove that functions $u \in G_-(\Omega; q, c, s, f)$ are locally bounded from below by considering $-u$.

\begin{theorem}
\label{thm:locB}
Let $1 < p \leq q < p^{\ast}$, $s_0 \in (0,1)$, $c_0, c_1 > 0$ and assume $s \in [s_0, 1)$. Let $f: [0, \infty) \to [0, \infty)$ be a convex increasing function satisfying \eqref{eq:pq-upper} and \eqref{eq:non-degeneracy}. If $u \in G_+(\Omega; q, c_1, s, f)$, then $u$ is locally bounded from above. Moreover, for each $B_{2R}(x_0) \subset \Omega$, there exists $C > 0$ such that for every $\delta \in (0,1)$
\begin{equation*}
\sup_{B_{R}(x_0)} u \leq \delta \mathrm{Tail}_{f'}(u_+;x_0, R) + C \delta^{-(q-1)\frac{p^{\ast}}{p} \frac{1}{p^{\ast}-q}} \left( \fint_{B_R(x_0)} u_+^q(x) \dx \right)^{\frac{1}{p}\frac{p^{\ast}-p}{p^{\ast}-q}} + \delta^{\frac{q-1}{q}}.
\end{equation*}
The constant $C$ depends on $d$, $s_0$, $p$, $q$, $p^{\ast}-q$, $c_0$, $c_1$ and $R$.
\end{theorem}

\begin{proof}
Let $x_0 \in \Omega$ and $R > 0$ be such that $B_{2R}(x_0) \subset \Omega$. We assume without loss of generality that $x_0 = 0$. For $j=0,1,\dots$, let
\begin{equation*}
R_j = (1+2^{-j})R, \quad k_j = (1-2^{-j})k, \quad \text{and}\quad \tilde{k}_j = (k_j+k_{j+1})/2,
\end{equation*}
where $k$ is an arbitrary positive number that will be determined later. We define $w_j = (u-k_j)_+$, $\tilde{w}_j = (u-\tilde{k}_j)_+$, $A_{k,R}^+ = \{ x \in B_R : u(x) > k \}$, and
\begin{equation*}
Y_j = \fint_{B_{R_j}} w_j^q(x) \dx.
\end{equation*}
Since $u \in  G_+(\Omega;q,c_1,s,f)$, using the assumptions \eqref{eq:pq-upper} and \eqref{eq:non-degeneracy} we have
\begin{equation*}
\begin{split}
(1-s)[\tilde{w}_j]_{W^{s, p}(B_{R_{j+1}})}^p 
&\leq C 2^{qj} \left( \int_{A_{\tilde{k}_j, R_j}^+} \left( \frac{\tilde{w}_j(x)}{R^s} \right)^q \dx + |A_{\tilde{k}_j, R_j}^+| \right) \\
&\quad + C(1-s) 2^{(d+sq)j} \|\tilde{w}_j\|_{L^1(B_{R_j})} \int_{\Rd \setminus B_{R_{j+1}}} f'\left( \frac{\tilde{w}_j(y)}{|y|^s} \right) \frac{\dy}{|y|^{d+s}}\\
&=: J_1 + J_2
\end{split}
\end{equation*}
for some $C = C(d, q, c_0, c_1) > 0$. Since
\begin{equation*}
|A_{\tilde{k}_j, r_j}^+| \leq \frac{1}{(\tilde{k}_j - k_j)^{q}} \int_{A_{\tilde{k}_j, r_j}^+} w_j^{q}(x) \dx \leq C \left( \frac{2^{j}}{k} \right)^{q} Y_j
\end{equation*}
and $\tilde{w}_j \leq w_j$, by assuming $k \geq \delta^{\frac{q-1}{q}}$ we have
\begin{equation*}
J_1 \leq C \delta^{-(q-1)} 2^{2qj} Y_j
\end{equation*}
for some $C = C(d, q, c_0, c_1, R) > 0$. For $J_2$, we observe 
\begin{equation*}
(2^{-j-2}k)^{q-1} \tilde{w}_j = (\tilde{k}_j-k_j)^{q-1} \tilde{w}_j \leq w_j^q
\end{equation*}
and $\mathrm{Tail}_{f'}(\tilde{w}_j; 0, R_{j+1}) \leq \mathrm{Tail}_{f'}(u_+; 0, R)$. Fix $\delta \in (0,1)$ and assume
\begin{equation} \label{eq:k-tail}
k \geq \delta \mathrm{Tail}_{f'}(u_+; 0, R) + \delta^{\frac{q-1}{q}}.
\end{equation}
Then, using \eqref{eq:k-tail}, \eqref{eq:f-finv}, \eqref{eq:pq-growth} and $f(1) = 1$, we deduce
\begin{equation*}
\begin{split}
J_2
&\leq C 2^{(d+sq)j} \frac{R_j^d}{(2^{-j-2}k)^{q-1}} Y_j \frac{1}{R_{j+1}^{s}} f'\left( \frac{k}{\delta R_{j+1}^{s}} \right) \\
&\leq C 2^{(d+sq)j} \frac{2^{(q-1)j}}{k^{q-1}}\frac{(1+(\frac{k}{\delta R^s})^q)}{\frac{k}{\delta R^s}} Y_j \leq C2^{(d+2q)j} \delta^{-(q-1)} Y_j
\end{split}
\end{equation*}
for some constant $C = C(d, q, c_0, c_1,s_0) > 0$.
Combining the estimates of $J_1$ and $J_2$, we arrive at
\begin{equation*}
(1-s)[\tilde{w}_j]_{W^{s, p}(B_{R_{j+1}})}^p \leq C 2^{(d+2q)j} \delta^{-(q-1)} Y_j.
\end{equation*}

Using \Cref{thm:frac-sobolev} and the inequality $\tilde{w}_j^{p^{\ast}} \geq (k_{j+1}-\tilde{k}_j)^{p^{\ast}-q} w_{j+1}^{q}$, we deduce 
\begin{equation*}
\begin{split}
(k_{j+1}-\tilde{k}_j)^{(p^{\ast}-q)p/p^{\ast}} Y_{j+1}^{p/p^{\ast}}
&\leq C \|\tilde{w}_j\|_{L^{p^{\ast}}(B_{R_{j+1}})}^p \\
&\leq C \left( \|\tilde{w}_j\|_{L^p(B_{R_{j+1}})}^p + (1-s) [\tilde{w}_j]_{W^{s, p}(B_{R_{j+1}})}^p \right) \\
&\leq C 2^{(d+2q)j} \delta^{-(q-1)} Y_j,
\end{split}
\end{equation*}
or
\begin{equation*}
Y_{j+1} \leq C k^{q-p^{\ast}} \delta^{-(q-1)p^{\ast}/p} b^j Y_j^{1+\beta},
\end{equation*}
where $b = 2^{p^{\ast}-q+(d+2q)p^{\ast}/p}$ and $\beta = \frac{p^{\ast}}{p}-1$. If $Y_0 \leq (C k^{q-p^{\ast}} \delta^{-(q-1)p^{\ast}/p})^{-1/\beta} b^{-1/\beta^2}$, then $Y_j \to 0$ as $j \to \infty$. Thus, if we assume
\begin{equation} \label{eq:k-Lq}
k^{p^{\ast}-q} \geq C \delta^{-(q-1)\frac{p^{\ast}}{p}} b^{1/\beta} Y_0^{\beta},
\end{equation}
then
\begin{equation*}
\sup_{B_{R}} u \leq k.
\end{equation*}
We now take
\begin{equation*}
k = \delta \mathrm{Tail}_{f'}(u_+;0, R) + C_0 \delta^{-(q-1)\frac{p^{\ast}}{p} \frac{1}{p^{\ast}-q}} \left( \fint_{B_R} u_+^q(x) \dx \right)^{\frac{1}{p}\frac{p^{\ast}-p}{p^{\ast}-q}} + \delta^{\frac{q-1}{q}}
\end{equation*}
with $C_0 = (Cb^{1/\beta})^{1/(p^{\ast}-q)}$, which is in accordance with \eqref{eq:k-tail} and \eqref{eq:k-Lq}.
\end{proof}

\section{Application to minimizers} \label{sec:minimizer}

In this section, we prove \Cref{thm:minimizer} by showing that local minimizers of \eqref{eq:nonlocalfunctional} belong to the De Giorgi class $G(\Omega; q, c, s, f)$ under some assumptions on $f$ and using several results from previous sections.\\
Let us first define local minimizers of \eqref{eq:nonlocalfunctional}. We assume that $f$ is a convex increasing function and $k$ satisfies \eqref{eq:k}, i.e.,
\begin{align}\tag{$k$}
k(x,y) = k(y,x) \quad\text{and}\quad \Lambda^{-1} \leq k(x,y) \leq \Lambda \quad\text{for a.e. } x, y \in \Rd.
\end{align}

\begin{definition} [minimizer]
We say that $u \in V^{s,f}(\Omega | \Rd)$ is a {\it local subminimizer} ({\it superminimizer}) of \eqref{eq:nonlocalfunctional} if for every measurable function $v : \Rd \to \R$ with $v = u$ a.e. in $\Rd \setminus \Omega$ and $v \le u$ ($v \ge u$) a.e. in $\Omega$, it holds that $\mathcal{I}_f(u) \le \mathcal{I}_f(v)$. We call $u \in V^{s, f}(\Omega|\Rd)$ a {\it local minimizer} of \eqref{eq:nonlocalfunctional} if it is a subminimizer and a superminimizer.
\end{definition}

Recall that we always assume $f(0) = 0$ and $f(1)=1$. This assumption can be made without loss of generality since $u$ minimizes $\mathcal{I}_f$ if and only if $u$ minimizes $\mathcal{I}_{(f - f(0))/(f(1)-f(0))}$.

\begin{theorem} \label{thm:DG-minimizer}
Let $s \in (0,1)$, $q > 1$ and $\Lambda \geq 1$. Let $f: [0, \infty) \to [0, \infty)$ be a convex increasing function satisfying \eqref{eq:pq-upper} and let $k: \Rd \times \Rd \to \R$ be a measurable function satisfying \eqref{eq:k}. Let $u \in V^{s,f}(\Omega|\Rd)$ be a local subminimizer of \eqref{eq:nonlocalfunctional}. Then, $u \in G_+(\Omega;q,c,s,f)$ for some $c = c(d, q, \Lambda)> 0$.
\end{theorem}

\begin{proof}
We follow the strategy carried out in \cite[Proposition 7.5]{Coz17}. Let $x_0 \in \Omega$, $0 < r < R \leq d(x_0, \partial\Omega)$ and $k \in \R$. Without loss of generality, we take $x_0 = 0$. Let $r \le \rho < \tau \le R$ and let $\eta \in C_c^{\infty}(\R^d)$ be a cutoff function with $0 \le \eta \le 1$, $\supp(\eta) = B_{\frac{\tau+\rho}{2}}$, $\eta \equiv 1$ in $B_{\rho}$, and $\Vert \nabla \eta \Vert_{\infty} \le \frac{4}{\tau-\rho}$. Let $w_\pm(x)$ and $A_k^\pm$ be as in \Cref{def:DG}. We define $A_{k,R}^+ = \{ x \in B_R : u(x) > k \}$ and $A_{k,R}^- = \{ x \in B_R : u(x) < k \}$.

We set $v := u - \eta^qw_+$, then $u \equiv v$ in $\Rd \setminus B_{\tau}$ and $u \geq v$ a.e. in $\Rd$. Since $u$ is a local subminimizer of \eqref{eq:nonlocalfunctional}, it holds that $\mathcal{I}_f(u) \le \mathcal{I}_f(v)$, i.e.,
\begin{equation} \label{eq:min}
0 \le (1-s) \int_{B_{\tau}} \int_{\Rd} A(x,y) \frac{k(x,y)}{|x-y|^{d}} \dy \dx,
\end{equation}
where
\begin{equation*}
A(x,y) = f\left(\frac{\vert v(x) - v(y)\vert}{\vert x-y \vert^s}\right) -  f\left(\frac{\vert u(x) - u(y)\vert}{\vert x-y \vert^s}\right).
\end{equation*}
To estimate $A(x,y)$, we distinguish four different cases and prove the following.

If $x \in A_k^-$ or $y \in A_k^-$, then
\begin{equation}\label{eq:Cacc1}
 A(x,y)\le 0.
\end{equation}
Furthermore, if $x,y\in A_{k,\rho}^+$, then
\begin{equation}\label{eq:Cacc2}
A(x,y) = - f\left(\frac{\vert w_+(x) - w_+(y)\vert}{\vert x-y \vert^s}\right). 
\end{equation}
If $x \in A_{k,\rho}^+$ and $y  \in A_{k}^-$, then
\begin{equation}
\label{eq:Cacc3}
A(x,y) \le -\frac{1}{2}\left[f'\left( \frac{w_-(y)}{\vert x-y \vert^s} \right) \frac{w_+(x)}{\vert x-y \vert^s} + f\left(\frac{\vert w_+(x) - w_+(y)\vert}{\vert x-y \vert^s}\right)\right].
\end{equation}
Finally, if $x,y \in A_k^+$, then we have
\begin{equation} \label{eq:Cacc4}
A(x,y) \le f\left(\frac{\vert w_+(x) - w_+(y)\vert}{\vert x-y \vert^s}\right) + cf\left(\frac{\vert \eta(x) - \eta(y)\vert (w_+(x) \vee w_+(y))}{\vert x-y \vert^s}\right)
\end{equation}
for some $c = c(q)>0$, where $a\vee b := \max\{a,b\}$.
In the following, we prove \eqref{eq:Cacc1}--\eqref{eq:Cacc4}. 
The proof of \eqref{eq:Cacc1} is a direct consequence of monotonicity of $f$. Namely, if $x \not\in A_k^+$, then
\begin{equation*}
\vert v(x) - v(y) \vert = \vert (1-\eta^q(y))w_+(y) + w_-(x)\vert \le \vert w_+(y) + w_-(x) \vert = \vert u(x) - u(y) \vert.
\end{equation*}
To see \eqref{eq:Cacc2}, observe that for $x,y \in A_{k,\rho}^+$ it holds that $\eta(x) = \eta(y) = 1$, and therefore
\begin{equation*}
\vert v(x) - v(y) \vert = \vert u(x) - u(y) - w_+(x) + w_+(y) \vert = 0.
\end{equation*}
Let us prove \eqref{eq:Cacc3}. For $x \in A_{k,\rho}^+$ and $y \not\in A_k^+$ it holds that
\begin{align*}
\vert v(x) - v(y) \vert &= \vert (1-\eta^q(x))w_+(x) + w_-(y) \vert = w_-(y), \\
\vert u(x) - u(y) \vert &= w_+(x) + w_-(y).
\end{align*}
By application of \Cref{lemma:convexlemma} with $\theta = \frac{1}{2}$, $a = \frac{w_-(y)}{\vert x-y \vert^s}$ and $b = \frac{w_+(x)}{\vert x-y \vert^s}$, we obtain
\begin{equation*}
A(x,y) \le -\frac{1}{2}\left[f'\left( \frac{w_-(y)}{\vert x-y \vert^s} \right) \frac{w_+(x)}{\vert x-y \vert^s} + f\left(\frac{w_+(x)}{\vert x-y \vert^s}\right)\right],
\end{equation*}
which implies \eqref{eq:Cacc3} since $w_+(y) = 0$.

To prove \eqref{eq:Cacc4} let us take $x,y \in A_{k}^+$. We compute
\begin{align*}
\vert v(x) - v(y) \vert &= \vert (1-\eta^q(x))w_+(x) -(1-\eta^q(y)) w_+(y) \vert\\
&= \vert (1-\eta^q(x))(w_+(x)-w_+(y)) +(\eta^q(y)-\eta^q(x)) w_+(y) \vert.
\end{align*}
Let us assume without loss of generality that $\eta(x) \ge \eta(y)$. Then, we have $\vert \eta^q(y)-\eta^q(x) \vert \le q \eta^{q-1}(x) \vert \eta(y)-\eta(x) \vert$. We estimate, using monotonicity and convexity of $f$, as well as \eqref{eq:pq-upper}
\begin{align*}
f\left(\frac{\vert v(x) - v(y)\vert}{\vert x-y \vert^s}\right) &\le f\left((1-\eta^q(x))\frac{\vert w_+(x)-w_+(y)\vert}{\vert x-y \vert^s} + \eta^{q}(x)\frac{q w_+(y) \vert \eta(y)-\eta(x)\vert}{ \eta(x)\vert x-y \vert^s}\right) \\
&\le (1-\eta^q(x))f\left(\frac{\vert w_+(x)-w_+(y)\vert}{\vert x-y \vert^s}\right) + \eta^{q}(x)f \left(\frac{q w_+(y) \vert \eta(y)-\eta(x)\vert}{\eta(x) \vert x-y \vert^s}\right)\\
&\le f\left(\frac{\vert w_+(x)-w_+(y)\vert}{\vert x-y \vert^s}\right) + q^q f \left(\frac{\vert \eta(y)-\eta(x)\vert}{\vert x-y \vert^s}w_+(y) \right),
\end{align*}
which implies \eqref{eq:Cacc4}.

By putting together the information from \eqref{eq:Cacc1}--\eqref{eq:Cacc3} and using assumptions on $k$, we deduce
\begin{equation} \label{eq:ball-rho}
\begin{split}
&(1-s) \int_{B_{\rho}} \int_{B_{\rho}} A(x, y) \frac{k(x,y)}{|x-y|^{d}} \dy \dx \\
&\le - \frac{1}{2\Lambda} \Phi_{W^{s, f}(B_\rho)}(w_+) - \frac{1}{2\Lambda}(1-s) \int_{B_{\rho}}\int_{A_{\rho, k}^-} f'\left( \frac{w_-(y)}{\vert x-y \vert^s} \right) \frac{w_+(x)}{\vert x-y \vert^s} \frac{\dy \dx}{\vert x-y \vert^{d}} .
\end{split}
\end{equation}
Moreover, from \eqref{eq:Cacc1}, \eqref{eq:Cacc3} and \eqref{eq:Cacc4} we obtain
\begin{equation} \label{eq:ball-tau}
\begin{split}
&(1-s) \iint_{B_{\tau}^2 \setminus B_{\rho}^2} A(x,y) \frac{k(x,y)}{|x-y|^d} \dy \dx \\
&\le \Lambda (1-s) \iint_{B_{\tau}^2 \setminus B_{\rho}^2} f\left(\frac{\vert w_+(x) - w_+(y)\vert}{\vert x-y \vert^s}\right) \frac{\dy \dx}{\vert x-y \vert^{d}}  \\
&\quad +c\Lambda(1-s) \iint_{B_{\tau}^2 \setminus B_{\rho}^2} f\left(\frac{\vert \eta(x) - \eta(y)\vert (w_+(y) \vee w_+(x))}{\vert x-y \vert^s}\right) \frac{\dy \dx}{\vert x-y \vert^{d}}  \\
&\quad - \frac{1}{2\Lambda} (1-s) \int_{B_{\rho}}\int_{(B_{\tau} \setminus B_{\rho}) \cap A_k^-} f'\left( \frac{w_-(y)}{\vert x-y \vert^s} \right) \frac{w_+(x)}{\vert x-y \vert^s} \frac{\dy \dx}{\vert x-y \vert^{d}} .
\end{split}
\end{equation}
Note that by monotonicity of $f$ and \Cref{lemma:upper}, we have
\begin{align*}
&(1-s)\int_{B_R} \int_{B_R}f\left(\frac{\vert \eta(x) - \eta(y)\vert}{\vert x-y \vert^s}w_+(x)\right) \vert x-y \vert^{-d} \dy \dx\\
&\le (1-s)\int_{B_R} \int_{B_{2R}(x)} f\left(\frac{4\vert x - y\vert^{1-s}}{\tau-\rho}w_+(x)\right) \vert x-y \vert^{-d} \dy \dx\\
&\le (1-s)\left(\frac{4R}{\tau-\rho}\right)^{q} \int_{B_R} \sum_{k = 0}^{\infty}\int_{B_{2^{-k+1}R}(x) \setminus B_{2^{-k}R}(x)} f\left((2^{-k+1})^{1-s}\frac{w_+(x)}{R^s}\right) (2^{-k}R)^{-d} \dy \dx\\
&\le c(1-s)\left(\frac{R}{\tau-\rho}\right)^{q} \int_{B_R} \sum_{k = 0}^{\infty} f\left(2^{-k(1-s)} \frac{w_+(x)}{R^s} \right) \dx
\end{align*}
for some $c = c(d, q) > 0$. We use \Cref{lem:conv} and \Cref{lemma:lower} to obtain
\begin{align*}
(1-s)\sum_{k = 0}^{\infty} f\left(2^{-k(1-s)} \frac{w_+(x)}{R^s} \right) 
&\le (1-s) \sum_{k=0}^{\infty} 2^{-k(1-s)} f\left(\frac{w_+(x)}{R^s}\right) \\
&\le \frac{1-s}{1-2^{-(1-s)}} f\left( \frac{w_+(x)}{R^s} \right).
\end{align*}
Since the map $s \mapsto (1-s)/(1-2^{-(1-s)})$ is bounded on $(0, 1)$ from above, we have
\begin{equation} \label{eq:ball-R}
(1-s)\int_{B_R} \int_{B_R}f\left(\frac{\vert \eta(x) - \eta(y)\vert}{\vert x-y \vert^s}w_+(x)\right) \vert x-y \vert^{-d} \dy \dx \leq c\left(\frac{R}{\tau-\rho}\right)^{q} \Phi_{L^f(B_R)}\left( \frac{w_+}{R^s}\right)
\end{equation}
for some $c = c(d, q) > 0$.

By combination of the estimates \eqref{eq:ball-rho}, \eqref{eq:ball-tau} and \eqref{eq:ball-R}, we get:
\begin{equation} \label{eq:local}
\begin{split}
&(1-s) \int_{B_{\tau}}\int_{B_{\tau}} A(x, y) \frac{k(x,y)}{|x-y|^{d}} \dy \dx \\
&\le c (1-s) \iint_{B_{\tau}^2 \setminus B_{\rho}^2} f\left(\frac{\vert w_+(x) - w_+(y)\vert}{\vert x-y \vert^s}\right) \frac{\d y \dx}{|x-y|^d} + c \left(\frac{R}{\tau-\rho} \right)^q \Phi_{L^f(B_R)}\left( \frac{w_+}{R^s}\right) \\
&\quad - c \Phi_{W^{s, f}(B_\rho)}(w_+) - c(1-s) \int_{B_{\rho}}\int_{A_{\tau, k}^-} f'\left( \frac{w_-(y)}{\vert x-y \vert^s} \right) \frac{w_+(x)}{\vert x-y \vert^s} \frac{\dy \dx}{\vert x-y \vert^{d}} .
\end{split}
\end{equation}

Let us deduce two more estimates for $A(x,y)$:
\begin{align}
\label{eq:Cacc5}
A(x,y) &= 0 \quad \text{for } x,y \not\in A_{k,\frac{\tau+\rho}{2}}^+,\\
\label{eq:Cacc6}
A(x,y) &\le f'\left( \frac{w_+(y)}{\vert x-y \vert^s} \right) \frac{w_+(x)}{\vert x-y \vert^s} \quad\text{for } x \in A_k^+ \text{ and } y \in A_k^+ \setminus B_{\tau}.
\end{align}
The proof of \eqref{eq:Cacc5} is trivial since $\supp(\eta) = B_{\frac{\tau+\rho}{2}}$ and $w_+=0$ on $A_k^-$. To see \eqref{eq:Cacc6}, we compute
\begin{align*}
A(x, y) = f\left(\frac{\vert (1-\eta^q(x))w_+(x) - w_+(y)\vert}{\vert x-y \vert^s}\right) -  f\left(\frac{\vert w_+(x) - w_+(y)\vert}{\vert x-y \vert^s}\right)
\end{align*}
and apply \Cref{lemma:CaccHelpLemma} with $\mu = 1-\eta^q(x)$, $a = \frac{w_+(x)}{\vert x-y \vert^s}$ and $b=\frac{w_+(y)}{\vert x-y \vert^s}$.
Consequently, by \eqref{eq:Cacc1}, \eqref{eq:Cacc3}, \eqref{eq:Cacc5} and \eqref{eq:Cacc6}, it holds that
\begin{equation} \label{eq:mixed}
\begin{split}
&(1-s) \int_{B_{\tau}} \int_{\Rd \setminus B_{\tau}} A(x, y) \frac{k(x,y)}{|x-y|^{d}} \dy \dx \\
&\le \Lambda (1-s) \int_{B_{\frac{\tau+\rho}{2}}} \int_{\Rd \setminus B_{\tau}}  f'\left( \frac{w_+(y)}{\vert x-y \vert^s} \right) \frac{w_+(x)}{\vert x-y \vert^s} \frac{\dy \dx}{\vert x-y \vert^{d}} \\
&\quad -\frac{1}{2\Lambda}(1-s)\int_{B_{\rho}}\int_{(\Rd \setminus B_{\tau}) \cap A_{k}^{-}} f'\left( \frac{w_-(y)}{\vert x-y \vert^s} \right) \frac{w_+(x)}{\vert x-y \vert^s} \frac{\dy \dx}{\vert x-y \vert^{d}} .
\end{split}
\end{equation}
Moreover, using \eqref{eq:derivativedoubling1} with $p=1$ we observe that
\begin{equation} \label{eq:mixed-tail}
\begin{split}
&(1-s) \int_{B_{\frac{\tau+\rho}{2}}} \int_{\Rd \setminus B_{\tau}}  f'\left( \frac{w_+(y)}{\vert x-y \vert^s} \right) \frac{w_+(x)}{\vert x-y \vert^s} \frac{\dy \dx}{\vert x-y \vert^{d}} \\
&\leq q(1-s) \left( \frac{2R}{\tau-\rho} \right)^{d+sq} \int_{B_{\frac{\tau+\rho}{2}}} \int_{\Rd \setminus B_{\tau}} f'\left( \frac{w_+(y)}{\vert y \vert^s} \right) \frac{w_+(x)}{|y|^{s}} |y|^{-d} \dy \dx \\
&\le c(1-s) \left(\frac{R}{\tau - \rho} \right)^{d+sq} \Vert w_+ \Vert_{L^1(B_{R})} \int_{\Rd \setminus B_{r}} f' \left( \frac{w_+(y)}{\vert y \vert^s} \right) \vert y \vert^{-d-s} \dy.
\end{split}
\end{equation}

By combining estimates \eqref{eq:min}, \eqref{eq:local}, \eqref{eq:mixed}, and \eqref{eq:mixed-tail}, we derive
\begin{align*}
&\Phi_{W^{s, f}(B_\rho)}(w_+) + (1-s) \int_{B_{\rho}}\int_{A_{k}^{-}} f'\left( \frac{w_-(y)}{\vert x-y \vert^s} \right) \frac{w_+(x)}{\vert x-y \vert^s} \frac{\dy \dx}{\vert x-y \vert^{d}} \\
&\le c \left(\Phi_{W^{s,f}(B_\tau)}(w_+) - \Phi_{W^{s,f}(B_\rho)}(w_+) \right) + c \left( \frac{R}{\tau-\rho} \right)^q \Phi_{L^f(B_R)} \left( \frac{w_+}{R^s} \right) \\
&\quad + c(1-s) \left(\frac{R}{\tau - \rho} \right)^{d+sq} \Vert w_+ \Vert_{L^1(B_{R})} \int_{\Rd \setminus B_{r}} f' \left( \frac{w_+(y)}{\vert y \vert^s} \right) \vert y \vert^{-d-s} \dy
\end{align*}
for some $c = c(d, q, \Lambda) > 0$. By setting
\begin{align*}
\phi(\rho) = \Phi_{W^{s, f}(B_\rho)}(w_+) + (1-s) \int_{B_{\rho}}\int_{A_{k}^{-}} f'\left( \frac{w_-(y)}{\vert x-y \vert^s} \right) \frac{w_+(x)}{\vert x-y \vert^s} \frac{\dy \dx}{\vert x-y \vert^{d}} ,
\end{align*}
we can deduce from the above line that 
\begin{align*}
\phi(\rho) 
&\le c(\phi(\tau)-\phi(\rho)) + c \left( \frac{R}{\tau-\rho} \right)^q \Phi_{L^f(B_R)} \left( \frac{w_+}{R^s} \right) \\
&\quad + c(1-s) \left(\frac{R}{\tau - \rho} \right)^{d+sq} \Vert w_+ \Vert_{L^1(B_{R})} \int_{\Rd \setminus B_{r}} f' \left( \frac{w_+(y)}{\vert y \vert^s} \right) \vert y \vert^{-d-s} \dy.
\end{align*}
We ``fill the hole" by adding $c\phi(\rho)$ to both sides. After dividing by $1+c$, we get that
\begin{equation} \label{eq:hole-filling-minimizer}
\begin{split}
\phi(\rho)
&\le \gamma \phi(\tau) + c \left( \frac{R}{\tau-\rho} \right)^q \Phi_{L^f(B_R)} \left( \frac{w_+}{R^s} \right) \\
&\quad + c(1-s) \left(\frac{R}{\tau - \rho} \right)^{d+sq} \Vert w_+ \Vert_{L^1(B_{R})} \int_{\Rd \setminus B_{r}} f' \left( \frac{w_+(y)}{\vert y \vert^s} \right) \vert y \vert^{-d-s} \dy,
\end{split}
\end{equation}
where $\gamma \in (0,1)$ and $c=c(d, q, \Lambda) > 0$. The desired result follows now from a standard iteration argument, see Lemma 4.11 in \cite{Coz17}.
\end{proof}

\begin{remark}
Similar to the proof of \Cref{thm:DG-minimizer}, it is possible to show that local superminimizer (minimizers, respectively) $u \in V^{s,f}(\Omega|\Rd)$ satisfies $u \in G_-(\Omega;q,c,s,f)$ ($u \in G(\Omega;q,c,s,f)$, respectively) for some $c = c(d, q, \Lambda) > 0$.
\end{remark}

\begin{proof} [Proof of \Cref{thm:minimizer}]
By \Cref{thm:DG-minimizer}, it follows that $u \in G(\Omega;q,c_1,s,f)$ for some $c_1 = c_1(d, q, \Lambda)> 0$. According to \Cref{prop:DG-F}, it holds that $u \in G(\Omega;q/p,c_2,s,g)$ for some $c_2 = c_2(d, p, q, \Lambda) > 0$. From  \Cref{thm:Fholder} and \Cref{thm:locB} we deduce the desired result.
\end{proof}

\section{Application to weak solutions} \label{sec:weaksoln}

In this section we aim to study weak solutions to nonlocal equations \eqref{eq:nonlocalequation} and prove \Cref{thm:weaksol}. Throughout this section we assume that $f$ is a convex increasing function satisfying \eqref{eq:pq-upper} and $h$ is a measurable function satisfying the structure condition \eqref{eq:h}, i.e.,
\begin{align}\tag{$h$}
h(x, y, t) = h(y, x, t), \quad \sign(t)\frac{1}{\Lambda} f'(|t|) \leq h(x,y,t) \leq \Lambda f'(|t|)
\end{align}
for a.e. $x, y \in \Rd$ and for all $t \in \R$.
We define weak solutions to \eqref{eq:nonlocalequation} as follows:
\begin{definition} [weak solution]
We say that $u \in V^{s,f}(\Omega | \Rd)$ is a {\it weak subsolution} to \eqref{eq:nonlocalequation} if for every $\phi \in V^{s,f}(\Omega|\Rd)$ with $\phi = 0$ a.e. in $\Rd \setminus \Omega$ and $\phi \ge 0$ a.e. in $\Omega$, it holds that
\begin{equation} \label{eq:weak-formulation}
(1-s) \iint_{(\Omega^c\times \Omega^c)^c} h \left( x, y, \frac{u(x) - u(y)}{\vert x-y \vert^s}\right) \frac{\phi(x)-\phi(y)}{\vert x-y \vert^{d+s}} \dy \dx \le 0.
\end{equation}
We say that $u \in V^{s,f}(\Omega | \Rd)$ is a {\it weak supersolution} if $-u$ is a weak subsolution. A function $u \in V^{s,f}(\Omega | \Rd)$ is called a {\it weak solution} if it is a weak subsolution and a weak supersolution.
\end{definition}

Recall that we always assume $f(0) = 0$ and $f(1)=1$. This assumption can be made without loss of generality since $u$ solves $\mathcal{L}_h u = 0$ if and only if $u$ solves $\mathcal{L}_{h/f(1)} u = 0$ and one can always choose $f(0) = 0$.

\begin{remark}
Let us prove that the weak formulation \eqref{eq:weak-formulation} is well-defined. Let $u,\phi \in V^{s,f}(\Omega | \Rd)$. Then, by \eqref{eq:h} and Fenchel's inequality \eqref{eq:Fenchel}, we have
\begin{align*}
&(1-s) \iint_{(\Omega^c\times \Omega^c)^c} \left\vert h \left( x, y, \frac{u(x) - u(y)}{\vert x-y \vert^s}\right) \frac{\phi(x)-\phi(y)}{\vert x-y \vert^{d+s}} \right \vert \dy \dx \\ 
&\leq \Lambda (1-s) \iint_{(\Omega^c\times \Omega^c)^c} f' \left( \frac{\vert u(x) - u(y) \vert}{\vert x-y \vert^s}\right) \frac{|\phi(x)-\phi(y)|}{\vert x-y \vert^{s}} \frac{\dy \dx}{\vert x-y \vert^{d}}  \\
&\le \Lambda(1-s) \iint_{(\Omega^c\times \Omega^c)^c} \left[ f^{\ast}\left(f' \left( \frac{\vert u(x) - u(y) \vert}{\vert x-y \vert^s}\right)\right) + f\left( \frac{\vert \phi(x) - \phi(y) \vert}{\vert x-y \vert^s}\right) \right] \frac{\dy \dx}{\vert x-y \vert^{d}} \\
&\le \Lambda (q-1) \Phi_{V^{s,f}(\Omega|\Rd)}(u) + \Lambda \Phi_{V^{s,f}(\Omega|\Rd)}(\phi) < \infty,
\end{align*}  
where we used that by \eqref{eq:legendre} and \eqref{eq:pq-upper}: $f^{\ast}(f'(t)) \le (q-1) f(t)$.
\end{remark}

The following theorem yields that weak solutions to \eqref{eq:nonlocalequation} belong to the De Giorgi classes introduced in \Cref{sec:degiorgiclasses}. 

\begin{theorem}
\label{thm:CaccSol}
Let $s \in (0,1)$, $q > 1$ and $\Lambda \geq 1$. Let $f: [0, \infty) \to [0, \infty)$ be a convex increasing function satisfying \eqref{eq:pq-upper} and let $h: \Rd \times \Rd \times \R \to \R$ be a measurable function satisfying \eqref{eq:h}. Let $u \in V^{s,f}(\Omega|\Rd)$ be a weak subsolution to \eqref{eq:nonlocalequation}. Then, $u \in G_+(\Omega;q,c,s,f)$ for some $c > 0$ depending on $d$, $q$ and $\Lambda$.
\end{theorem}

\begin{proof}
The desired result follows from a similar argument as in the proof of \Cref{thm:DG-minimizer}. Let $x_0 \in \Omega$, $0 < r < R \leq d(x_0, \partial\Omega)$, and $k \in \R$. We may assume without loss of generality that $x_0 = 0$. Let $r \leq \rho < \tau \leq R$ and let $\eta \in C_c^{\infty}(\Rd)$ be a cutoff function with $0 \le \eta \le 1$, $\supp(\eta) = B_{\frac{\tau+\rho}{2}}$, $\eta \equiv 1$ in $B_{\rho}$, and $\Vert \nabla \eta \Vert_{\infty} \leq \frac{4}{\tau-\rho}$. We define $w_\pm(x)$, $A_k^{\pm}$ and $A_{k, R}^\pm$ as in \Cref{thm:DG-minimizer}. We set $v = \eta^q w_+$. Since $u$ is a weak subsolution to \eqref{eq:nonlocalequation}, we have
\begin{equation} \label{eq:subsoln}
0 \geq (1-s) \iint_{(\Omega^c\times \Omega^c)^c} B(x, y) \frac{\dy \dx}{|x-y|^{d}} ,
\end{equation}
where
\begin{equation*}
B(x,y) = h\left( x, y, \frac{u(x)-u(y)}{|x-y|^s} \right) \frac{v(x)-v(y)}{|x-y|^s}.
\end{equation*}

Let us estimate $B(x, y)$. If $x, y \in A_k^-$, then
\begin{equation} \label{eq:DG1}
B(x,y) = 0.
\end{equation}
If $x \in A_{k}^+$ and $y \in A_{k}^-$, then by \eqref{eq:h}
\begin{equation*}
h\left( x, y, \frac{u(x)-u(y)}{|x-y|^{s}} \right) = h\left( x, y, \frac{w_+(x)+w_-(y)}{|x-y|^{s}} \right) \geq \frac{1}{\Lambda} f' \left( \frac{w_+(x)+w_-(y)}{|x-y|^{s}} \right).
\end{equation*}
Thus, we obtain
\begin{equation} \label{eq:DG2}
\begin{split}
B(x,y)
&\geq \frac{1}{\Lambda} f'\left( \frac{|w_+(x)+w_-(y)|}{|x-y|^s} \right) \frac{w_+(x)}{|x-y|^s} \eta^q(x) \\
&\geq \frac{1}{2\Lambda} \left[ f'\left( \frac{w_+(x)}{|x-y|^s} \right) + f'\left( \frac{w_-(y)}{|x-y|^s} \right) \right] \frac{w_+(x)}{|x-y|^s} \eta^q(x) \\
&\geq \frac{1}{2\Lambda} f\left( \frac{|w_+(x)-w_+(y)|}{|x-y|^s} \right) \eta^q(x) + \frac{1}{2\Lambda} f'\left( \frac{w_-(y)}{|x-y|^s} \right) \frac{w_+(x)}{|x-y|^s} \eta^q(x),
\end{split}
\end{equation}
where we used \eqref{eq:der-subadd} and \Cref{lem:conv}.

If $x, y \in A_{k}^+$, we prove
\begin{equation} \label{eq:DG3}
\begin{split}
B(x,y) 
&\geq \frac{1}{\Lambda} f\left( \frac{|w_+(x)-w_+(y)|}{|x-y|^s} \right)(\eta^q(x) \lor \eta^q(y)) - \varepsilon \Lambda(q-1) f\left( \frac{|w_+(x)-w_+(y)|}{|x-y|^s} \right) \\
&\quad - c\Lambda f\left( \frac{w_+(x) \lor w_+(y)}{|x-y|^s} |\eta(x)-\eta(y)| \right)
\end{split}
\end{equation}
for any $\varepsilon \in (0,1)$, where $c = c(q, \varepsilon) > 0$. It is enough to prove \eqref{eq:DG3} for the case $w_+(x) \geq w_+(y)$. If $\eta(x) \geq \eta(y)$, then \eqref{eq:DG3} follows from
\begin{equation*}
B(x,y) \geq h\left( x, y, \frac{w_+(x)-w_+(y)}{|x-y|^s} \right) \frac{w_+(x)-w_+(y)}{|x-y|^s} \eta^q(x) \geq \frac{1}{\Lambda} f\left( \frac{w_+(x)-w_+(y)}{|x-y|^s} \right) \eta^q(x),
\end{equation*}
where we used \eqref{eq:h} and \Cref{lem:conv}. When $\eta(y) \geq \eta(x)$, we observe that
\begin{equation} \label{eq:eta1}
\begin{split}
B(x, y)
&= h\left( x, y, \frac{w_+(x)-w_+(y)}{|x-y|^s} \right) \left( \frac{w_+(x)-w_+(y)}{|x-y|^s} \eta^q(y) - \frac{w_+(x)}{|x-y|^s} (\eta^q(y)-\eta^q(x)) \right) \\
&\geq \frac{1}{\Lambda} f \left( \frac{w_+(x)-w_+(y)}{|x-y|^s} \right) \eta^q(y) - \Lambda f' \left( \frac{w_+(x)-w_+(y)}{|x-y|^s} \right) \frac{w_+(x)}{|x-y|^s} (\eta^q(y)-\eta^q(x)),
\end{split}
\end{equation}
where we used \Cref{lem:conv} again. Note that
\begin{equation*}
\eta^q(y) - \eta^q(x) \leq q\eta^{q-1}(y) (\eta(y)-\eta(x)) \leq q (\eta(y)-\eta(x)).
\end{equation*}
Thus, for $\varepsilon \in (0,1)$ we use \Cref{lemma:upper} and the Fenchel's inequality \eqref{eq:Fenchel} to obtain
\begin{align*}
&f' \left( \frac{w_+(x)-w_+(y)}{|x-y|^s} \right) \frac{w_+(x)}{|x-y|^s} (\eta^q(y)-\eta^q(x)) \\
&\leq q \varepsilon^{-q} f' \left( \varepsilon \frac{w_+(x)-w_+(y)}{|x-y|^s} \right) \frac{w_+(x)}{|x-y|^s} (\eta(y)-\eta(x)) \\
&\leq f^{\ast}\left( f'\left(\varepsilon \frac{w_+(x)-w_+(y)}{|x-y|^s} \right) \right) + f\left(q\varepsilon^{-q}  \frac{\eta(y)-\eta(x)}{|x-y|^s} w_+(x) \right).
\end{align*}
By \eqref{eq:legendre}, \eqref{eq:pq-upper} and \Cref{lemma:lower} (iv) with $p=1$, we deduce that
\begin{equation} \label{eq:eta2}
\begin{split}
&f' \left( \frac{w_+(x)-w_+(y)}{|x-y|^s} \right) \frac{w_+(x)}{|x-y|^s} (\eta^q(y)-\eta^q(x)) \\
&\leq \varepsilon(q-1) f\left( \frac{w_+(x)-w_+(y)}{|x-y|^s} \right) + c f\left( \frac{\eta(y)-\eta(x)}{|x-y|^s} w_+(x) \right)
\end{split}
\end{equation}
for some $c=c(q, \varepsilon) > 0$. Therefore, \eqref{eq:DG3} follows from \eqref{eq:eta1} and \eqref{eq:eta2}.

Combining \eqref{eq:DG1}, \eqref{eq:DG2} and \eqref{eq:DG3}, we have
\begin{equation} \label{eq:DG-rho}
\begin{split}
&(1-s) \int_{B_{\tau}} \int_{B_{\tau}} B(x,y) \frac{\dy \dx}{|x-y|^{d}}  \\
&\geq \frac{1}{2\Lambda} \Phi_{W^{s, f}(B_{\rho})}(w_+) + \frac{1}{2\Lambda}(1-s) \int_{B_{\rho}} \int_{A_{k, \tau}^-} f'\left( \frac{w_-(y)}{|x-y|^s} \right) \frac{w_+(x)}{|x-y|^s} \frac{\dy \dx}{|x-y|^{d}}  \\
&\quad - \varepsilon \Lambda(q-1) \Phi_{W^{s, f}(B_{\tau})}(w_+) - c \Lambda(1-s) \int_{B_{\tau}} \int_{B_{\tau}} f \left( \frac{|\eta(x)-\eta(y)|}{|x-y|^s} w_+(x) \right) \frac{\d y \dx}{|x-y|^{d}}.
\end{split}
\end{equation}

Let us now take into account the pairs $(x,y) \in (\Rd \times \Rd) \setminus (B_{\tau} \times B_{\tau})$. Using \eqref{eq:h} we compute
\begin{equation} \label{eq:DG-I12}
\begin{split}
&(1-s)\int_{B_{\tau}} \int_{\Rd \setminus B_{\tau}} B(x, y) \frac{\dy \dx}{|x-y|^{d}}  \\
&\ge \frac{1}{\Lambda}(1-s) \int_{A_{k,\rho}^+} \int_{\{ u(x) > u (y)\} \setminus B_{\tau}} f'\left( \frac{\vert u(x) - u(y)\vert}{\vert x-y \vert^s} \right) \frac{w_+(x)}{|x-y|^{s}} \frac{\dy \dx}{\vert x-y \vert^{d}} \\
&\quad - \frac{1}{\Lambda}(1-s) \int_{A_{k,\frac{\tau+\rho}{2}}^+} w_+(x) \int_{\{ u(y) > u (x)\} \setminus B_{\tau}} f'\left( \frac{\vert u(x) - u(y)\vert}{\vert x-y \vert^s} \right) \frac{\dy \dx}{\vert x-y \vert^{d+s}} =: I_1 - I_2.
\end{split}
\end{equation}
By monotonicity of $f'$, we have that 
\begin{equation} \label{eq:DG-I1}
I_1 \ge \frac{1}{2\Lambda}(1-s) \int_{B_{\rho}} \int_{(\Rd \setminus B_{\tau}) \cap A_k^-} f'\left( \frac{ w_-(y)}{\vert x-y \vert^s} \right) \frac{w_+(x)}{|x-y|^{s}} \frac{\dy \dx}{\vert x-y \vert^{d}}.
\end{equation}
Moreover, by \eqref{eq:mixed-tail} we obtain
\begin{equation} \label{eq:DG-I2}
I_2 \le c (1-s) \left( \frac{R}{\tau-\rho} \right)^{d+sq} \Vert w_+ \Vert_{L^1(B_R)} \int_{\Rd \setminus B_{\tau}} f'\left(\frac{w_+(y)}{\vert y \vert^s} \right) \frac{\dy}{\vert y \vert^{d+s}}.
\end{equation}

Therefore, it follows from \eqref{eq:subsoln}, \eqref{eq:DG-rho}--\eqref{eq:DG-I2} and \eqref{eq:ball-R} that
\begin{equation} \label{eq:hole-filling-weaksol}
\begin{split}
&\frac{1}{2\Lambda} \Phi_{W^{s, f}(B_\rho)}(w_+) + \frac{1}{2\Lambda} (1-s) \int_{B_{\rho}}\int_{A_{k}^{-}} f'\left( \frac{w_-(y)}{\vert x-y \vert^s} \right) \frac{w_+(x)}{\vert x-y \vert^s} \frac{\dy \dx}{\vert x-y \vert^{d}}  \\
&\le \varepsilon\Lambda(q-1) \Phi_{W^{s,f}(B_\tau)}(w_+) + c \left( \frac{R}{\tau-\rho} \right)^q \Phi_{L^f(B_R)} \left( \frac{w_+}{R^s} \right) \\
&\quad + c(1-s) \left(\frac{R}{\tau - \rho} \right)^{d+sq} \Vert w_+ \Vert_{L^1(B_{R})} \int_{\Rd \setminus B_{r}} f' \left( \frac{w_+(y)}{\vert y \vert^s} \right) \frac{\dy}{\vert y \vert^{d+s}},
\end{split}
\end{equation}
where $c = c(d, q,\Lambda, \varepsilon) > 0$. We take $\varepsilon = (4\Lambda^2(q-1))^{-1} \wedge 2^{-1} \in (0,1)$ so that \eqref{eq:hole-filling-weaksol} boils down to \eqref{eq:hole-filling-minimizer} with $\gamma=1/2$ and $c = c(d,q,\Lambda) > 0$. This finishes the proof.
\end{proof}

\begin{remark}
Similar to the proof of \Cref{thm:CaccSol}, it is possible to show that weak solutions $u \in V^{s,f}(\Omega|\R^d)$ to $\mathcal{L}_h u = 0$ in $\Omega$ satisfy $u \in G(\Omega;q,c,s,f)$ for some $c > 0$ depending on $d$, $q$ and $\Lambda$.
\end{remark}

\begin{proof} [Proof of \Cref{thm:weaksol}]
By \Cref{thm:CaccSol}, it follows that $u \in G(\Omega;q,c_1,s,f)$ for some $c_1 > 0$ depending on $d$, $q$ and $\Lambda$. According to \Cref{prop:DG-F}, it holds that $u \in G(\Omega;q,c_2,s,g)$ for some $c_2 = c_2(d, p, q, \Lambda) > 0$. From \Cref{thm:Fholder} and \Cref{thm:locB} we deduce the desired result.
\end{proof}


\end{document}